\documentclass[10pt]{elsarticle}
\usepackage{amsmath}
\usepackage{amsfonts}
\usepackage{amssymb}
\usepackage{graphicx}
\usepackage{eso-pic}
\usepackage{dsfont}
\usepackage{color}
\usepackage{ulem}
\usepackage{wrapfig}

\newtheorem{theorem}{Theorem}[section]

\newtheorem{corollary}[theorem]{Corollary}

\newdefinition{definition}{Definition}[section]

\newtheorem{lem}[theorem]{Lemma}

\newtheorem{proposition}[theorem]{Proposition}
\newdefinition{remark}[theorem]{Remark}

\newenvironment{proof}[1][Proof]{\noindent\textsc{#1.} }{$\Box$}

\newcommand{\ind}[1]{\mathds{1}_{\{#1\}}}
\newcommand{\R}{\mathbb{R}}

\renewcommand{\L}{\mathrm{L}}
\renewcommand{\d}{\,\mathrm{d}}
\newcommand{\ds}{\displaystyle}
\newcommand{\supp}{\mathop{\rm supp}\,}

\newcommand{\e}{\,\mathrm{e}\,}

\newcommand{\dy}{\,\mathrm{d}y}

\renewcommand{\d}{\,\mathrm{d}}
\newcommand{\dist}{\mathrm{dist}}

\newcommand{\eps}{\varepsilon}
\renewcommand{\leq}{\leqslant}
\renewcommand{\geq}{\geqslant}

\newcommand{\pe}[2]{(#1\cdot#2)}

\definecolor{mauve}{rgb}{0.4,0,0.4}
\newcommand{\mauve}[1]{{\color{mauve} #1}}
\let\margintemp=\marginpar
\renewcommand{\marginpar}[1]{\margintemp{\tiny\textbullet\;\mauve{#1}}}

\journal{}

\newcommand{\tr}{\mathop{\rm Tr}}
\newcommand{\Hess}{H^{\rm ess}}
\newcommand{\Less}{L^{\rm ess}}
\newcommand{\ang}{\mathop{\rm ang}}

\begin{document}

\begin{frontmatter}



\title{Large Deviations estimates for some non-local equations. General bounds and applications}


\author[cb]{C. Br\"andle}
\ead{cbrandle@math.uc3m.es}
\author[ec]{E. Chasseigne\corref{cor1}}
\ead{echasseigne@univ-tours.fr}
\address[cb]{Departamento de Matem{\'a}ticas, U.~Carlos III de Madrid,
28911 Legan{\'e}s, Spain}
\address[ec]{Laboratoire de Math\'ematiques et Physique Th\'eorique, U.~F. Rabelais, Parc de Grandmont, 37200 Tours, France\\
}

\cortext[cor1]{Corresponding author}

\begin{abstract}
Large deviation estimates for the following  linear parabolic equation are studied:
\[
\frac{\partial u}{\partial t}=\tr\Big( a(x)D^2u\Big) + b(x)\cdot D u +
	\int_{\R^N} \Big\{(u(x+y)-u(x)-(D u(x)\cdot y)\ind{|y|<1}(y)\Big\}\d\mu(y)\,,
\]
where $\mu$ is a L\'evy measure (which may be singular at the origin).
Assuming only that some negative exponential integrates with respect to the tail of $\mu$, it is shown that
given an initial data, solutions defined in a bounded domain converge exponentially fast to the
solution of the problem defined in the whole space. The exact rate, which depends strongly on
the decay of $\mu$ at infinity, is also estimated.

\

\textbf{R\'esum\'e}\\[12pt]
Nous donnons ici des estimations de grande d\'eviations pour l'\'equation parabolique lin\'eaire suivante:
\[
\frac{\partial u}{\partial t}=\tr\Big( a(x)D^2u\Big) + b(x)\cdot D u +
	\int_{\R^N} \Big\{(u(x+y)-u(x)-(D u(x)\cdot y)\ind{|y|<1}(y)\Big\}\d\mu(y)\,,
\]
o\`u $\mu$ est une mesure de L\'evy (possiblement singuli\`ere \`a l'origine).
En supposant seulement qu'une exponentielle n\'egative est int\'egrable par rapport \`a la
queue de distribution de la mesure $\mu$, nous montrons que,
\'etant fix\'ee une donn\'ee initiale, les solutions d\'efinies dans un domaine born\'e convergent
vers la solution du probl\`eme dans l'espace tout entier \`a un taux exponentiel.
Le taux de convergence exact, qui d\'epend fortement du comportement de $\mu$ \`a l'infini, est \'egalement
estim\'e.
\end{abstract}

\begin{keyword}
Non-local diffusion, Large deviation \sep Hamilton-Jacobi equation \sep L\'evy operators.

\MSC 47G20 \sep 60F10 \sep 35A35 \sep 49L25
\end{keyword}

\end{frontmatter}

\

\section{Introduction}

The equations we consider in this paper take the following form:
\begin{equation}\label{eq:0}
\frac{\partial u}{\partial t}=\tr\Big( a(x)D^2u\Big) + b(x)\cdot D u +
	\int_{\R^N} \Big\{(u(x+y)-u(x)-(D u(x)\cdot y)\ind{|y|<1}(y)\Big\}\d\mu(y)\,.
\end{equation}
Recently, equations like \eqref{eq:0} involving L\'evy-type non-local terms have been under
thorough investigation by many authors and in many directions. Indeed, they present challenging problems which are not covered by the existing ``local'' vision of pde's.
References are too many to cite, those which are more closely related to our concern are given below.

Our main goal here is to obtain some estimates on how the solutions $u_R$ that are defined
in the ball $B_R=\{|x|<R\}$ approach the solution in $\R^N$ as $R\to\infty$, in the spirit of \cite{BarlesDaherRomano94}.
To sum up briefly, the main results of the present paper show that  provided $\mu$ has at most
an exponential tail, a bound of the following kind always holds as $R\to\infty$:
\begin{equation}\label{est:0}
	|u-u_R|(x,t)\leq \e^{-R f(\ln R)(1 + o(1))}.
\end{equation}
The behaviour of $f$ depends on the tail of $\mu$, but in any case $f$ cannot grow faster than linearly.
This has to be compared to the classical $\exp(-R^2)$-type estimate associated to the heat equation \cite{BarlesDaherRomano94}: the presence of non-local terms implies an order $\exp(-R\ln R)$ at best.

\

\noindent\textbf{Typical examples -} As for \eqref{eq:0}, a first example we have in mind is the convolution equation,
\begin{equation}\label{eq:1}
	\frac{\partial u}{\partial t}=J\ast u -u.
\end{equation}
Here $a,b=0$ and $\mu$ is a L\'evy measure with probability density $J$, which is symmetric and integrable -- see \cite{ChasseigneChavesRossi07,Chasseigne07}. In \cite{BrandleChasseigne08-1},
the authors gave partial answers for this equation when $J$ decays strictly faster than an exponential at infinity. Our results here include and generalize those given in \cite{BrandleChasseigne08-1}. Be carefull that writing the equation
under the form \eqref{eq:1} requires $J$ to be of unit mass, and moreover that in this case $\mu(z)=J(-z)$,
a detail which has to be taken into account when dealing with non-symmetric kernels.

A more sophisticated example is the following:
$$\frac{\partial u}{\partial t}=\Delta u +b\cdot Du +\mathrm{P.V.}\,
\int_{\R^N}\Big\{u(x+y,t)-u(x,t)\Big\}\frac{\e^{-|y|}}{|y|^{N+\alpha}}\dy\,,$$
where P.V. stands for \textit{principal value}. This is the same as putting the \textit{compensator} term $Du(x)\cdot y$
since the measure is symmetric. The measure $J(y)=\e^{-|y|}/|y|^{N+\alpha}$, called
\textit{tempered $\alpha$-stable law}, appears for instance in finance \cite{OksendalSulem}.

It is important
to notice that our framework (which will be fully established in Section~\ref{sect:preliminaries}) does not allow to treat fractional Laplace-type nonlocal terms for which $J(y)=1/|y|^{N+\alpha}$,
$\alpha\in(0,2)$ since we require the tail of $J$ to go to zero at least exponentially fast.

\

\textbf{Existence and uniqueness of solutions -}
Concerning the existence and uniqueness of solutions of~\eqref{eq:0}, first notice that in the case of constant coefficients
$a$ and $b$, this can be derived from a Fourier analysis of the equation.
In the absence of differential terms, i.e., $a=b=0$ and if $\mu$ is not singular at the origin, various results about
existence and uniqueness of solutions in $\R^N$ or bounded domains can be found in
\cite{BrandleChasseigne08-1,BrandleChasseigneFerreira09,ChasseigneChavesRossi07,Chasseigne07}.

Finally, in the case when $a(x)=\sigma\sigma^t$ with $\sigma$ and $b$ Lipschitz, existence and uniqueness of bounded solutions
are proven in \cite{BarlesImbert07} for the elliptic (stationnary) version of the equation under some assumptions
on $\mu$. With little modifications (which essentially amount to changing the $(-\gamma u)$-term in the equation
for a $u_t$-term), the same result holds true for the parabolic version,
and gives existence and uniqueness of solutions in $\R^N$ or in a bounded domain.
See also \cite{AlvarezTourin96,CaffarelliSilvestre07,CaffarelliSilvestre09} for some other results.

We shall not derive here a full theory of existence, uniqueness and comparison for~\eqref{eq:0} which is not the
central point of this paper. Rather, we will simply assume that given an initial data $u_0$ continuous, positive
and bounded, there exists a unique viscosity solution $u(x,t)$ of~\eqref{eq:0} with initial data $u_0$. We also consider for any $R>0$ the solution $u_R$ of~\eqref{eq:0} in $B_R$ with initial data $u_0$ and $u(x,t)=0$ for any $x>R$ (see~\cite{BrandleChasseigne08-1} for details) and assume that for any $R>0$, there exists a unique viscosity
solution $u_R$. Starting from this, we will
now estimate the difference $(u-u_R)$.

\

\textbf{Large deviations -}
Estimate \eqref{est:0} may be seen as a \textit{large deviations} result if one considers the probabilistic viewpoint associated to
the equation. We shall not enter here into details about L\'evy processes and refer for instance to \cite{BertoinBook,ContBook} for a probability-oriented approach, but let us just mention a few facts.

A L\'evy process is a stochastic process $(Y_t)_{t\geq0}$ which has stationary and independent increments:
for any $0<s<t$, $(Y_t-Y_s)$ only depends on $(t-s)$ and is independent of $(Y_{t'}-Y_{s'})$ for non-overlapping time intervals.
An important fact is that $(Y_t)$ is not required to be continuous (in time), 
so that this type of process can model either continuous diffusions or
jump diffusions, and a mix of both. Thus, typical L\'evy processes are the ``usual'' brownian motion, 
the compound Poisson process and the $\alpha$-stable jump process.

Now, if $u(\cdot,t)$ represents the density of the law of $Y_t$, then 
it can be shown, using the L\'evy-Khintchine formula, that the \textit{characteristic function} of $Y_t$, which is
the Fourier transform of $u$, takes the following form
$$\hat{u}(\xi,t)=\e^{t\varphi(\xi)}\,,$$
where the \textit{characteristic exponent} $\varphi$ is the sum of
of $(i)$ a brownian motion, $(ii)$ a drift and $(iii)$ a pure jump process. More precisely, in Fourier variables:
$$\varphi(\xi)= A \xi\cdot\xi+B\cdot\xi+\int\Big(\e^{\mathrm{i}(\xi\cdot y)} -1-\mathrm{i}(\xi\cdot y)\ind{|y|<1}\Big)\d\mu(y)\,,$$
where $\mu$ is a L\'evy measure.
By taking the inverse Fourier transform of $\hat{u}_t=\varphi(\xi)\hat{u}$, we recover exactly equation \eqref{eq:0}: 
Brownian motions are associated to
Laplacian-type diffusion terms $\tr\big( a(x)D^2u\big)$ while a drift
corresponds to the $b(x)\cdot D u$-term. In the case of the compound Poisson process, 
$$\varphi(\xi)=\int \Big(\e^{\mathrm{i}(\xi\cdot y)}-1\Big)J(y)\dy\,,$$ 
where $J\in\L^1$ is the jump distribution of the process so that the non-local term
can be re-written as $J\ast u$, leading to \eqref{eq:1}. 
Finally, in the case of the $\alpha$-stable law $\mu(z)=|z|^{-N-\alpha}$, we recover
the well-known fractional Laplace operator.

In some sense, estimate \eqref{est:0} measures how many processes have escaped the ball $B_R$ between
time $t=0$ and $t=T$ (we refer to \cite{BarlesDaherRomano94} for a precise proof of this assertion
in the case of the Laplacian),
and obtaining an exponential-type estimate is precisely what is called a large deviation result -- see the
introduction of \cite{BrandleChasseigne08-1}, and \cite{Hollander} for more about this.

\

\textbf{Notations - } Throughout the paper we will denote by $(e_i)$ the canonical basis of $\R^N$, while $\nu$ will refer to any unit vector. The euclidean scalar product is denoted as $a\cdot b$ and we use $|\cdot |$ for the associated norm in $\R^N$. We will use $\ang(a,b)$ to denote the angle between the vectors $a$ and $b$ (see~Definition~\ref{def:angle}) and $a\to a^T$ is the transposition. The notation $B(x,r)$ stands, as is usual, for the ball of radius $r>0$, centered at $x\in\R^N$ and we also use the simpler notation $B_r=B(0,r)$. We also denote $Du\in\R^N$ the gradient of $u$ and $D^2u\in\mathcal{M}_N(\R)$ the Hessian matrix of $u$. In all the paper, $c(J)$ denotes any constant which only depends on $J$, possibly varying from one line to another. The notation $J_1\curlyeqprec J_2$ essentially means that $J_1\leq J_2$ in $\R^N$, but we add some strict separation inside a small region (see Definition \ref{def:ess.ordered}). We denote by $a\wedge b$ the infimum of $a$ and $b$.

\

\textbf{Main results - }
We have first a general result, Theorem \ref{thm:est-IR}, which gives a theoretical bound in terms of a rate function $I_\infty$, typical of large deviation estimates:
as $R\to\infty$,
\begin{equation*}
	|u-u_R|(x,t)\leq \e^{-R I_\infty(x/R,t/R)+o(1)R}\,.
\end{equation*}
The rate function $I_\infty$ satisfies the limit Hamilton-Jacobi problem
\begin{equation}
    \label{eq:Iinfty}
		\begin{cases}
			\partial_t I_\infty + H(D I_\infty)=0 & \text{in }B_1\times(0,\infty),\\
			I_\infty=0&\text{on }\partial B_1\times(0,\infty),\\
			I_\infty(x,0)=+\infty & \text{in }B_1,
		\end{cases}
\end{equation}
where the Hamiltonian associated to \eqref{eq:0} is given by:
\begin{equation}
	\label{eq:hamiltonian:full}
	H(p):=\tr\big(A\,p\otimes p) + B\cdot p+ \int_{\R^N}\Big(\e^{p\cdot y}-1-(p\cdot y)\ind{|y|<1}(y)\Big)\d\mu(y)\,.
\end{equation}
The idea of the proof is to suitably rescale the problem for $u-u_R$ (which is not the same as in the case of
the local heat equation, though) and to derive an equation for the log-transform;
then pass to the limit as $R\to\infty$.

The Lagrangian $L$ associated to $H$ is defined as usual as
\begin{equation}
  \label{eq:lagrangian}
  L(q):=\sup_{p\in\R^N}\big\{p\cdot q-H(p)\big\}\,,
\end{equation}
and using a Lax-Oleinik formula, see~\cite{LionsBook}, we obtain a semi-explicit expression for $I_\infty$, see~\eqref{est:IR} in terms of $L$.

In our present framework, obtaining \eqref{eq:Iinfty} is more involved than in~\cite{BrandleChasseigne08-1}, since we have to deal with the differential terms and the singular part in $H$.
Hence, we face several difficulties that come from the generalization of the equation, which imply non trivial adaptations and new developments of some aspects of the theory. Let us briefly explain the main points we have to deal with:
\begin{enumerate}
\item  \textit{Singular measures at the origin:}
	using kernels with singularities requires first a suitable concept of solution;
    we use here the notion of viscosity solutions derived in \cite{BarlesChasseigneImbert07} for L\'evy-type operators,
    which allows us to handle this situation.
    Then we introduce what we call the \textit{essential Hamiltonian}, $\Hess$, wiping out the
	singularity at the origin, see Definition~\ref{def:ess:hamiltonian}, and prove that
	asymptotically $H$ and $\Hess$ are equivalent (and the same holds for the associated Lagrangians).

\item \textit{Non symmetric kernels:} we want to estimate $L$ in order to obtain an expression for $I_\infty$. This needs an extra effort in this case: since we do not have symmetry, the vectors $p$ and $DH(p)$ do not necessarily point in the same direction. Hence dot products can not be simplified and we have to carefully analyze the angle between vectors. We prove some results valid for non symmetric kernels $J$ at the logarithm level, and others using the smallest symmetric kernel above $J$.

\item \textit{Differential terms:} the first two terms in equation~\eqref{eq:0} do not introduce much difficulties with respect to~\cite{BrandleChasseigne08-1} but we have to deal with them with the notion of viscosity solution. To be able to pass
    to the limit, we will assume that $a$ and $b$ have limits at infinity.

\item \textit{Possibly infinite Hamiltonians:} in the case when $J(y)\sim \e^{-\alpha |y|}$ as $|y|\to\infty$, $H$ is infinite  outside $B_\alpha$, which makes the analysis of~\eqref{eq:Iinfty} much more delicate, since an initial layer appears. Thus, one of the major contributions of this paper is to provide existence, uniqueness and comparison results for the Hamilton-Jacobi equation $u_t+H(Du)=0$ when $H$ its infinite outside the a ball.

\end{enumerate}



\textbf{Organization -}
After a preliminary section, Section~\ref{sect:preliminaries}, in which some properties of the generalized non-local equation as well as for the  Hamiltonian are stated, we devote Section~\ref{sect:theoretical} to the theoretical behaviour of the problem, where we prove Theorem~\ref{thm:est-IR}.
Then we concentrate our efforts in finding a behaviour for $L$, that allows us to give more explicit convergence rates for $u-u_R$. We deal with this issue in sections~\ref{sect:compact},~\ref{sect:intermediate} and~\ref{sect:critical}, where we consider different types of measures (compactly supported, intermediate, exponentially decaying), which lead to different behaviours. Section~\ref{sect:critical} especially deals with critical kernels for which $H$ is not finite everywhere; so we devote a big part of this section to treat the Hamilton-Jacobi equation with such $H$. Finally in Section~\ref{sect:comments} we collect some further results: we give some explicit calculations of rates and consider totally asymmetric measures in 1-D. We also explain how our methods allow to obtain some estimates for the related nonlocal KPP-equation.

%
%
%
%
%

\section{Preliminaries}\label{sect:preliminaries}
\setcounter{equation}{0}

Let us focus now on the concrete hypothesis  that concern the  equations we have in mind.
For $\mu,a,b,A,B$ in~\eqref{eq:0} and~\eqref{eq:hamiltonian:full} there are essentially three assumptions we shall make throughout the paper:

\textbf{Hypothesis 1.} The measure $\mu$ is a L\'evy measure with density $J(\cdot)$ satisfying
\begin{eqnarray}\label{eq:hyp:J}\begin{cases}
	J:\R^N\setminus\{0\}\to\R_+\text{ is }\mathrm{C}^1\text{-smooth},\\[2mm]
	\exists\rho_0>0\,,\ \supp(J)\supset B_{\rho_0},\\[2mm]
	\displaystyle\int_{\R^N} (1\wedge|y|^2)J(y)\dy<\infty\,.
	\end{cases}
\end{eqnarray}
 Hence in the sequel we rewrite~\eqref{eq:hamiltonian:full} as
\begin{equation}
  \label{eq:H:J}
  H(p):=\tr\big(A\,p\otimes p) + B\cdot p+ \int_{\R^N}\Big(\e^{p\cdot y}-1-(p\cdot y)\ind{|y|<1}(y)\Big)J(y)\dy\,.
\end{equation}
The last assumption in~\eqref{eq:hyp:J} implies in particular that the integral in \eqref{eq:H:J} over $\{|y|<1\}$ converges.

\textbf{Hypothesis 2.} The $N\times N$-matrix $a(\cdot)$ is continuous and nonnegative, $b(\cdot)$ is a continuous vector field in $\R^N$
and we also assume that the following limits are well-defined:
\begin{equation}\label{eq:limit:differentials}
	A=\lim_{|x|\to\infty} a(x)\in\mathcal{M}_n(\R)\,,\quad B=\lim_{|x|\to\infty} b(x)\in\R^N\,,
\end{equation}
(notice that necessarily $A$ is nonnegative).

Except in Section \ref{sect:critical}, we also make the following assumption on $J$:

\textbf{Hypothesis 3.}
\begin{equation}\label{eq:hyp2:J}
	\forall\beta>0\,,\quad \int_{\{|y|>\rho_0/2\}} \e^{\beta|y|}J(y)\dy<\infty\,.
\end{equation}
This ensures that the exponential part of the Hamiltonian in \eqref{eq:hamiltonian:full} (or~\eqref{eq:H:J})
converges for any $p\in\R^N$.
In Section \ref{sect:critical}, we shall only assume that
\begin{equation}\label{eq:hyp3:J}
	\sup\big\{\beta>0: \int_{\{|y|>\rho_0/2\}} \e^{\beta|y|}J(y)\dy<\infty\big\}=\beta_0>0\,,
\end{equation}
which is the case for exponential-type tails, and implies that the associated Hamiltonian is
not everywhere-defined.

Most of our results would be also valid in a more general setting, for instance we could relax
the $\mathrm{C}^1$ condition on $J$, but this is not a major matter for this paper. There are also some
specific other assumptions that we shall make in Section \ref{sect:intermediate}.


\subsection{Viscosity solutions of the equation}

Equations like \eqref{eq:0} can be treated by various ways, depending on the coefficients $a(\cdot)$ and
$b(\cdot)$ and whether the measure $\mu$ is singular
at the origin or not. Since we make only general assumptions, the best tool in order to
deal with differential terms and a singular measure is provided by the notion of \textit{viscosity solutions}.

Essentially, the notion of viscosity solution allows to give sense to the differential terms as well as
the integral term without knowing \textit{a priori} that the solution is regular, using comparison with smooth test functions.
Notice however that in the integral term, one only needs to replace $u$ by a smooth function in a small
neighborhood of $y=0$. We refer to \cite{BarlesImbert07,CaffarelliSilvestre07,CaffarelliSilvestre09} for more results on viscosity solutions in presence of singular measures. So, if we set
\[
    \begin{aligned}
	\mathrm{Eq}[u,\phi,\delta](x_0,t_0)\  :=&\  \frac{\partial \phi}{\partial t}(x_0,t_0)-\tr\Big( a(x)D^2\phi(x_0,t_0)\Big)
	- b(x_0)\cdot D \phi(x_0,t_0) \\
	-&\int_{\{|y|<\delta\}} \Big\{(\phi(x_0+y,t_0)-\phi(x_0,t_0)-(D \phi(x_0,t_0)\cdot y)\Big\}J(y)\dy\\
	-&\int_{\{|y|\geq\delta\}} \Big\{(u(x_0+y,t_0)-u(x_0,t_0)-(D \phi(x_0,t_0)\cdot y)\ind{|y|<1}\Big\}J(y)\dy,
\end{aligned}
\]
 we then introduce the following definition:

\begin{definition}
A locally bounded u.s.c. function $u : \R^N\times[0,\infty) \to \R $ is a subsolution of \eqref{eq:0} iff, for any $\delta\in(0,1)$ and any test function
$\phi\in C^2(\R^N\times\R_+)$, at each maximum point $(x_0,t_0)\in\R^N\times\R_+$ of $u-\phi$, we have
\[
    \mathrm{Eq}[u,\phi,\delta](x_0,t_0)\leq 0.
\]
A locally bounded l.s.c. viscosity supersolution is defined in the same way with reversed inequality and $\min$ instead of $\max$.
Finally, a viscosity solution is a function which the upper and lower semicontinuous envelopes are respectively sub- and supersolution of the problem.
\end{definition}

When intial and/or boundary data are involved, the definition takes into account that at a boundary point, either
the inequation has to hold, or the boundary data has to be taken in the sub/super solution sense. If $\Omega$ is the domain of definition, $f$ is the boundary data (in the sense of nonlocal equations) and $u(x,0)=u_0(x)$ is the initial data of the problem we define:

\begin{definition}\label{def:relaxed:boundary}
A locally bounded u.s.c. function $u : \R^N\times[0,\infty) \to \R $ is a subsolution of \eqref{eq:0} with boundary data $u=f$ in
$(\R^N\setminus\Omega)\times(0,\infty)$ and initial data $u(x,0)=u_0(x)$ iff,
for any $\delta>0$ and any test function $\phi\in C^2(\R^N\times(0,\infty))$, at each maximum point $(x_0,t_0)\in\Omega\times[0,\infty)$ of $u-\phi$, we have
\begin{eqnarray*}
(x_0,t_0)\in \Omega\times(0,\infty)&\Rightarrow&\mathrm{Eq}[u,\phi,\delta](x_0,t_0)\leq 0\\
(x_0,t_0)\in \partial\Omega\times(0,\infty)&\Rightarrow& \min\Big(\mathrm{Eq}[u,\phi,\delta](x_0,t_0)\, ;\, u(x_0,t_0) - f(x_0,t_0)\Big) \leq 0\\
(x_0,t_0)\in \Omega\times\{0\}&\Rightarrow& u(x_0,t_0)\leq u_0(x_0)
\end{eqnarray*}
A supersolution is defined with reversed inequalities and min/max changed accordingly,
and a solution is such that its l.s.c./u.s.c. envelopes are sub/super solutions of the equation.
\end{definition}

\subsection{Hamiltonians}

We shall show now some properties of the specific Hamiltonians we consider:

\begin{lem}\label{lem:hamiltonian:properties}
	Let $J$ be a kernel satisfying \eqref{eq:hyp:J}. Then
	the Hamiltonian $H$ is superlinear and strictly convex, and thus the associated Lagrangian
	$L$	is well-defined, convex and superlinear.
\end{lem}
\begin{proof}
	A straightforward calculus shows that in the sense of matrices,
	$$D^2H(p)=A + \int_{\R^N}(y\otimes y)\e^{p\cdot y}J(y)\dy>0$$
	so that $H$ is strictly convex. Now, one can easily check that
	$$H(p)=O(|p|^3)+\int_{\{|y|>1\}}\e^{p\cdot y}J(y)\dy\,.$$
	Indeed, the differential terms are clearly of the order of $|p|^2$ and $|p|$ respectively,
	while the integral over $|y|<1$ can be bounded by:
	$$\int_{\{|y|<1\}}\Big(\e^{p\cdot y}-1-(p\cdot y)\Big)J(y)\dy\leq
	C|p|^2\int_{\{|y|<1\}}\frac{|y|^2}{2}(1+O(p\cdot y))J(y)\dy\leq C|p|^3\,.$$
	On the other hand,
	$$\int_{\{|y|>1\}}\e^{p\cdot y}J(y)\dy\geq\int_{\{p\cdot y>\rho_0|p|/2\}}\e^{p\cdot y}J(y)\dy
	\geq c(J)\e^{\rho_0|p|/2}\,,$$
	so that indeed $H$ is superlinear. It is well-known (see~\cite{Rockafellar} for instance) that
    if $H$ is strictly convex and superlinear, so does $L$.
\end{proof}

We shall go a bit further in this direction, using that the main contribution of the Hamiltonian comes from
the exponential term in the integral. So let us define the \textit{essential part} of
the Hamiltonian and the corresponding Lagrangian:
\begin{definition}\label{def:ess:hamiltonian}
	We denote by
	\[
		\Hess(p):=\int_{\{|y|>\rho_0/2\}}\e^{p\cdot y}J(y)\dy\quad\text{and}\quad
		\Less(q):=\sup_{p\in\R^N}\big\{p\cdot q-\Hess(p)\big\}\,,
	\]
    where $\rho_0$ is defined in \eqref{eq:hyp:J}.
\end{definition}
The reason why we integrate over $\{|y|>\rho_0/2\}$ is that we want to avoid the singularity at the origin,
but we also want to be sure to integrate within the support of $J$ when it is compactly supported. The condition
$\supp(J)\supset B_{\rho_0}$ ensures that $\{|y|>\rho_0/2\}\cap\supp(J)\neq\emptyset$.

We have first a basic estimate, similar to the one that was used in Lemma \ref{lem:hamiltonian:properties}:

\begin{lem}
	\label{lemma:H:basic}
	Let $J$ be a kernel satisfying \eqref{eq:hyp:J} and \eqref{eq:hyp3:J}. Then we have,
	\begin{eqnarray*}	
	\Hess(p)&\geq& c(J)\e^{\rho_0|p|/2}\,,\\
	p\cdot D \Hess(p)&\geq& c(J)(\rho_0|p|/2)\e^{\rho_0|p|/2}+O(|p|)\,.
	\end{eqnarray*}
	Moreover there exists $\eps>0$ such that for any unit vector $\mathbf{\nu}$,
	$$\mathbf{\nu}^TD^2 \Hess(p)\,\mathbf{\nu}\geq \eps c(J)\e^{|p|\rho_0/2}$$
	and the same result holds for $D^2H=A+D^2\Hess\geq D^2\Hess$.
\end{lem}
\begin{proof}
	The first estimate is easily obtained:
	\begin{eqnarray*}
	\Hess(p)&\geq&\int_{\{p\cdot y>\rho_0|p|/2\}}\e^{p\cdot y}J(y)\dy\\
	&\geq&c(J)\e^{\rho_0|p|/2}\,.	
	\end{eqnarray*}
	 Notice that the constant $c(J)$ may be small,
	but  it is still positive since $J$ is continuous and positive inside $\{p\cdot y>\rho_0|p|/2\}$.

	For the second estimate, the proof is essentially the same:
	\begin{eqnarray*}
		p\cdot D \Hess(p)&\geq&\int_{\{p\cdot y>\rho_0|p|/2\}}(p\cdot y)\e^{p\cdot y}J(y)\dy+
			\int_{\{p\cdot y<0\}}(p\cdot y)\e^{p\cdot y}J(y)\dy\\
		&\geq&c(J)(\rho_0|p|/2)\e^{\rho_0|p|/2}+O(|p|)-c(J)|p|\,,	
	\end{eqnarray*}
	where in this case $c(J)=\int_{\{|y|>1\}}|y|J(y)\dy<\infty$ from \eqref{eq:hyp3:J}.
	
	Finally, for any unit vector $e_i$ of the canonical basis of $\R^N$, we compute
	with the same decomposition:
	\begin{eqnarray*}
		e_i^TD^2 \Hess(p)\,e_i=\int_{\{|y|>\rho_0/2\}}|y_i|^2\e^{p\cdot y}J(y)\dy&
		\geq&\int_{\{p\cdot y>\rho_0|p|/2\}}|y_i|^2\e^{p\cdot y}J(y)\dy\,.
	\end{eqnarray*}
	We notice that the set $\{p\cdot y>\rho_0|p|/2\}$ contains the truncated cone:
	$$\Big\{\frac{p\cdot y}{|p||y|}\geq\frac{3}{4}\,,\ |y|>\frac{3\rho_0}{4}\Big\}\,,$$
	hence for some $\eps>0$ small enough (which can be chosen independently of $e_i$),
	$$\mathcal{B}:=\{p\cdot y>\rho_0|p|/2\}\cap \{|y_i|^2\geq\eps\}\not=\emptyset\,,$$
	and since $\mathcal{B}\cap B_{\rho_0}\neq\emptyset$, the integral of $J$ on $\mathcal{B}$
	is positive. So we can estimate the integral from below as follows:
	\begin{eqnarray*}
	e_i^TD^2 \Hess(p)\,e_i&\geq&\eps\e^{|p|\rho_0/2}\int_{\mathcal{B}}J(y)\dy\\
	&\geq&\eps c(J)\e^{|p|\rho_0/2}\,.
	\end{eqnarray*}
    So the result also holds for any unit vector $\mathbf{\nu}$.
\end{proof}

The next estimate will be essential in the sequel:
\begin{lem}
 \label{lemma:H:less:pDHp}
	Let $J$ be a kernel satisfying \eqref{eq:hyp:J} and \eqref{eq:hyp2:J}. Then, for any $\gamma\in(0,\rho_0)$,
  	\begin{equation}\label{est:Hess}
  	\Hess(p)\leq \frac{c(J)}{\gamma}+\frac{p\cdot D\Hess(p)}{\gamma|p|}+\e^{\gamma|p|}\,.
  	\end{equation}
	As a consequence, $\Hess(p)$ is negligible in front of $p\cdot D\Hess(p)$ as $|p|\to\infty$,
	$$
	\lim_{|p|\to\infty}\frac{\Hess(p)}{p\cdot D\Hess(p)}=0\,.
	$$
\end{lem}

\begin{proof}
	We begin with splitting the integral in two terms,
  $$
 \Hess(p)=\int_{\{p\cdot y<\gamma|p|\}\cap \{|y|\geq \rho_0/2\}}\e^{p\cdot y}J(y)\dy+\int_{\{p\cdot y\geq\gamma|p|\}\cap \{|y|\geq \rho_0/2\}}\e^{p\cdot y}J(y)\dy=I_1+I_2
  $$
	The first integral is estimated as follows:
  $$
  I_1\leq \int_{\{p\cdot y<\gamma|p|\}\cap \{|y|\geq \rho_0/2\}}\e^{\gamma|p|}J(y)\dy\leq c(J)\e^{\gamma|p|}.
  $$
	Now, notice that since $\gamma<\rho_0$, even if $J$ is compactly
	supported, we are sure that the  integral $I_2$ concerns a region where $J$ is not zero, so that we can
	indeed control it by $p\cdot D \Hess(p)$:
  \[
  \begin{aligned}
  I_2\ \leq\ &\frac{1}{\gamma|p|}\int_{\{p\cdot y\geq\gamma|p|\}}(p\cdot y)\e^{p\cdot y}J(y)\dy\\
  =\ &\frac{1}{\gamma|p|}\int_{\mathbb{R}^N}(p\cdot y)\e^{p\cdot y}J(y)\dy-\frac{1}{\gamma|p|}
	\int_{\{p\cdot y\leq\gamma|p|\}}(p\cdot y)\e^{p\cdot y}J(y)\dy\\
  \leq\ & \frac{p\cdot D\Hess(p)}{\gamma|p|}-\frac{1}{\gamma|p|}\int_{\{p\cdot y\leq 0\}}(p\cdot y)\e^{p\cdot y}J(y)\dy-\frac{1}{\gamma|p|}
        \int_{\{0\leq p\cdot y\leq\gamma|p|\}}(p\cdot y)\e^{p\cdot y}J(y)\dy\\
  \leq\  & \frac{p\cdot D\Hess(p)}{\gamma|p|}+\frac{1}{\gamma|p|}\int_{\{p\cdot y\leq 0\}}|p||y|J(y)\dy\\
  \leq\  & \frac{p\cdot D\Hess(p)}{\gamma|p|}+ \frac{c(J)}{\gamma}.
  \end{aligned}
    \]

	Summing up $I_1$ and $I_2$ gives the first result of the lemma.

	For the second part we first fix $\mu>1$, $\gamma=\mu/|p|$ and choose $|p|>\mu/\rho_0$ so that $\gamma\in(0,\rho_0)$. Using estimate
    \eqref{est:Hess}, we obtain:
	$$
	\Hess(p)\leq \frac{c(J)|p|}{\mu}+\frac{p\cdot D\Hess(p)}{\mu}+\e^{\mu}\,,
	$$
	which implies
	$$\frac{p\cdot D\Hess(p)}{\Hess(p)}\geq \mu - \frac{c(J)|p|}{\Hess(p)}-\frac{\mu\e^\mu}{\Hess(p)}\,.$$
	Since $\Hess(p)\to\infty$ superlinearly (see Lemma \ref{lemma:H:basic}) we then have
	$$
	\liminf_{|p|\to\infty}\frac{p\cdot D\Hess(p)}{\Hess(p)}\geq \mu\,.
	$$
	But since $\mu>1$ is arbitrary, we conclude that the limit is $+\infty$, hence the result.
\end{proof}

 \begin{remark}
 {\rm
   If $D \Hess(p)$ would grow faster than an exponential, estimate \eqref{est:Hess} would be sufficient to conclude that
	$\Hess(p)$ is negligible in front of $p\cdot D\Hess(p)$. Actually this will be the case when $J$ is not compactly supported, see Lemma~{\rm\ref{lem:gen:est}}.}
 \end{remark}

The following lemma proves that essentially, the estimates we will produce in this paper do not depend
on $A,B$ and neither on the behaviour of $J$ near the origin.
\begin{lem}
	\label{lemma:H:singular}
	For any $J$ satisfying \eqref{eq:hyp:J} and \eqref{eq:hyp2:J},
	as $|p|\to\infty$, we have
	\[
		H(p)\sim \Hess(p)\,,\ DH(p)\sim D\Hess(p)\text{ and }L(p)\sim \Less(p)\,.
	\]
	Moreover, if $J$ is symmetric, then $\Hess$ and $\Less$ are also symmetric.
\end{lem}

\begin{proof}
	We split the Hamiltonian as follows:
	\[
    \begin{aligned}
    H(p)\ =&\ \tr\big(Ap\otimes p\big)+B\cdot p+\Hess(p)+\int_{\{|y|\leq \rho_0/2\}}\Big(\e^{p\cdot y}-1-(p\cdot y)\Big)J(y)\dy\\
	  -&\ \int_{\{\rho_0/2<|y|\leq 1|\}}\Big(1+(p\cdot y)\Big)J(y)\dy-\int_{\{|y|>1\}}J(y)\dy\\
	=&\ \Hess(p)+ \tr\big(Ap\otimes p\big)+B\cdot p +\int_{\{|y|\leq\rho_0/2\}}\Big(\frac{(p\cdot y)^2}{2}(1+O(p\cdot y))\Big)J(y)\dy+c(J)\\
	=&\ \Hess(p)+O(|p|^3)\,.
	\end{aligned}
    \]
	Since the Hamiltonians behave at least exponentially, Lemma~\ref{lemma:H:basic},
	the $O(|p|^3)$ is negligible by far, and $H$ and $\Hess$ are equivalent.
	The calculations for $DH$ and $D\Hess$ are similar: we get
	$$|DH(p)-D\Hess(p)|=O(|p|^2)\,,$$
	while $|DH(p)|$ grows at least exponentially.
	Notice that since $H$ and $\Hess$ are equivalent, Lemma \ref{lemma:H:less:pDHp} implies
	that both are negligible in front of $DH$ and $D\Hess$.

	More involved is the proof that $\Less(p)\sim L(p)$. Let us notice first that the point $p_0=p_0(q)$
	where $L(q)$ reaches its sup goes to infinity as $|q|\to\infty$. Indeed, this comes from the fact that
	$q=DH(p_0)$ and that $DH(p)$ grows at least exponentially, see Lemma \ref{lemma:H:basic}. The same holds for
	$D \Hess$.
	
	Then we proceed as follows:
	for $|q|$ large enough, since the sup below are attained for $|p|$ also large,
	there is some constant $C>0$ such that
	\begin{equation}\label{eq:comp:sup}
	\sup_{p\in\R^N}\{p\cdot y - \Hess(p)-C|p|^3\}\leq L(q)\leq\sup_{p\in\R^N}\{p\cdot y - \Hess(p)+C|p|^3\}\,.
	\end{equation}
	Hence if we denote by $p_1=p_1(q)$ the point  where the sup on the right is attained, we have
	$q=D \Hess(p_1)+3C|p_1|p_1$, while $q=D H(p_0)$ for the sup in $L(q)$. Now,
	$$ q=D\Hess(p_1)+3C|p_1|p_1=D\Hess(p_1)+O(|p_1|^2)=D H(p_1)+O(|p_1|^2)=DH(p_0)\,,$$
	and from this we will conclude that $p_1(q)\sim p_0(q)$. More precisely,
	$$O(|p_1|^2)=DH(p_1)-DH(p_0)=D^2H(\xi)\cdot (p_1-p_0)\,,$$
	for some $\xi$ in the segment $[p_0,p_1]$. We take the scalar product with $(p_1-p_0)$ and use Lemma~\ref{lemma:H:basic}
	to get an estimate from above:
	$$|p_1-p_0|^2\Big(\eps c(J)\e^{|\xi|\rho_0/2}\Big)\leq \pe{D^2H(\xi)\cdot (p_1-p_0)}{(p_1-p_0)}\leq C|p_1|^2|p_1-p_0|\,.$$
	From this follows that for some constant still noted $C>0$,
	$$|p_1-p_0|\leq C\frac{\e^{-|\xi|\rho_0/2}}{|p_1|^2}\,.$$
	Noticing that $|\xi|\geq\min\{|p_0|,|p_1|\}\to\infty$, this estimate implies that as $|q|\to\infty$,
	$$\frac{|p_1-p_0|}{\min\{|p_0|,|p_1|\}}\leq \frac{C}{|p_1|^2}
	\frac{\e^{-\min\{|p_0|,|p_1|\}\rho_0/2}}{\min\{|p_0|,|p_1|\}}\to0\,.$$
	Hence we have proved that $p_1(q)\sim p_0(q)$.

	The same calculation is valid for the sup on the left of \eqref{eq:comp:sup}, so that
	we conclude that
	\[
		\Less(q)=p_1\cdot D\Hess(p_1) - \Hess(p_1)\sim p_0\cdot D H(p_0)- H(p_0)=L(q)\,.
	\]
	About the symmetry see \cite[Lemma 2.4]{BrandleChasseigne08-1}.
\end{proof}

We finally prove a result which allows us to compare the Lagrangians when the kernels are ordered.
But for some technical reason, we need the kernels to be strictly ordered in $B_{\rho}\setminus B_{\rho_0/2}$,
hence we introduce the following notation:
\begin{definition}\label{def:ess.ordered}
	We say that $J_1$ and $J_2$ are essentially ordered and we denote $J_1\curlyeqprec J_2$ if there exist
	$a,b>0$ such that $\rho_0/2<a<b<\rho_0$ and
	$$J_1\leq J_2\text{ in }\R^N\,,\quad J_1(y)<J_2(y)\text{ in }\{a<|y|<b|\}\,.
	$$
\end{definition}

\begin{lem}
	\label{lemma:reduce:symmetric}
	Let us assume that $J_1,J_2$ satisfy~\eqref{eq:hyp:J} and \eqref{eq:hyp2:J} and that
	$J_1\curlyeqprec J_2$. Denote
	by $H_1,H_2,L_1,L_2$ the associated Hamiltonians and Lagrangians. Then for $p,q$ big enough:
	\[
	H_1(p)\leq H_2(p)\text{\quad and \quad} L_1(q)\geq L_2(q)\,.
	\]
\end{lem}

\begin{proof}
	The inequality $\Hess_1(p)\leq \Hess_2(p)$ is always true, since the integrand in $\Hess$ is always positive
	and $J_1\leq J_2$.	But as such, this is not enough to derive the same inequality between $H_1$ and $H_2$.
	To this end, let us notice that since $J_1\curlyeqprec J_2$,
	then there exists $\eps>0$ and $\rho_0/2<a<b<\rho_0$ such that
	$$ J_1(y)\leq J_2(y)-\eps\chi(y)\quad\text{in}\quad\R^N\,,$$
	where $ \chi(x)=\big((b-|x|)(|x|-a)\big)_+\,.$

	Then we shall estimate the difference $\Hess_1-\Hess_2$ as in Lemma~\ref{lemma:H:basic}:
	\begin{eqnarray*}
		\Hess_1(p)-\Hess_2(p)&\leq& -\eps\int_{\{a<|y|<b\}}\e^{p\cdot y}\chi(y)\dy\\
		&\leq& -\eps\int_{\{a<|y|<b\}\cap\{p\cdot y>a|p|\}}\e^{p\cdot y}\chi(y)\dy\\		
		&\leq& -\eps c(\chi)\e^{a|p|}\,,
	\end{eqnarray*}
	where $c(\chi)$ is the integral of $\chi$ over $\{a<|y|<b\}\cap\{p\cdot y>a|p|\}$.
	It is clear that this set has non-empty interior and since $\chi>0$  in this set, then
	$c(\chi)$ is a stictly positive constant. This proves that not only are $\Hess_1$ and $\Hess_2$ ordered,
	but moreover the difference between the two is at least of the order of an exponential.

	Then we use, as in the proof of Lemma \ref{lemma:H:singular}, that $H_1(p)=\Hess_1(p)+O(|p|^3)$
	and the same for $H_2$, so that
	\begin{eqnarray}
		H_1(p)-H_2(p)=\Hess_1(p)-\Hess_2(p)+O(|p|^3)\leq -\eps c(\chi)\e^{a|p|}  + O(|p|^3)\,.
	\end{eqnarray}
	Thus for $|p|$ big enough, we have indeed $H_1(p)\leq H_2(p)$.

The inequality for $L_1$ and $L_2$ follows by definition.
\end{proof}

\begin{remark}
	($i$) This last result is also valid if $J_1$ satisfies \eqref{eq:hyp3:J} with $\beta_1$,
	$J_2$ satisfies \eqref{eq:hyp3:J} with $\beta_2\leq\beta_1$ and we consider $|p|<\beta_2$.
	($ii$) Notice that if we only assume $J_1=J_2$ outside $B_\rho$ and $J_1<J_2$ in $B_\rho$, then
	$J_1\curlyeqprec J_2$. In this case, the tails of both kernels and their support
	(if compactly supported) remain exactly the same so that essentially,
	all the estimates we give in the sequel are the same for both kernels.
\end{remark}

\subsection{Viscosity solutions of the limit Hamilton-Jacobi equation}

In Section \ref{sect:theoretical}, we will have to study the following Hamilton-Jacobi equation with
Cauchy-Dirichlet boundary values,
\begin{equation}\label{pb:limitA}
		\begin{cases}
			\partial_t I^A + H(D I^A)=0 & \text{in }B_1\times(0,\infty),\\
			I^A=0&\text{on }\partial B_1\times(0,\infty),\\
			I^A(x,0)=A & \text{in }B_1.
		\end{cases}
\end{equation}
where we assume here that $H\in\mathrm{C}(\R^N;\R)$. We refer to Section \ref{sect:critical} for the case when
the Hamiltonian could take infinite values.
Let us first recall the definition of  viscosity solutions for this equation
(see for instance \cite{CrandallLions83}):
\begin{definition}
\label{def:viscosity}
	A locally bounded u.s.c function $u:\overline{B}_1\times[0,\infty)\to\mathbb{R}$
 is a viscosity subsolution of
	\eqref{pb:limitA} if for any $\mathrm{C}^2$-smooth function $\varphi$, and any point
	$(x_0,t_0)\in\overline{B}_1\times[0,\infty)$ where $u-\varphi$ reaches a maximum, there holds,
	\begin{eqnarray*}
		x_0\in B_1 &\Rightarrow& \partial_t\varphi+H(D \varphi)\leq0\text{ at }(x_0,t_0),\\
		x_0\in\partial B_1 &\Rightarrow & \min\big\{\partial_t\varphi+H(D \varphi)\,;\,u\big\}
		\leq0\text{ at }(x_0,t_0)\,,\\
		t_0=0 &\Rightarrow & u\leq A\text{ in }B_1\,.
	\end{eqnarray*}
A locally bounded l.s.c. function is a viscosity supersolution if the same holds with
reversed inequalities and min replaced by max at the boundary. Finally a viscosity solution is a
locally bounded function $u$ such that its u.s.c. and l.s.c. envelopes are respectively sub- and
super-solutions of \eqref{pb:limitA}.
\end{definition}

Notice that in general, the initial data as to be understood in the relaxed viscosity way (either
the data is taken, or the sub/super solution condition holds) but it is well-known (see for instance
\cite{BarlesBook}) that in the case of continuous Hamiltonian, this condition is equivalent to
the one we use here. Keep in mind, however, that this will not be the case in Section \ref{sect:critical}.

Since $H$ is convex we have the following representation:
\begin{lem}\label{lem:representation}
	If $J$ satisfies \eqref{eq:hyp:J} and~\eqref{eq:hyp2:J},
	then there exists a unique viscosity solution of~\eqref{pb:limitA}, which is given by
	\begin{equation*}
		I^A(x,t)=A\wedge \min_{y\in \partial B_1}
		\Big\{t\,L\Big(\frac{x-y}{t}\Big)\Big\}
		\quad\text{in}\quad B_1\times(0,\infty)\,.
	\end{equation*}
\end{lem}

\begin{proof}
	The assumptions made on $J$ imply that both $H$ and $L$ are finite everywhere,
	convex and super-linear so that uniqueness holds for this problem, see~\cite{LionsBook}.
	Then the Lax-Oleinik formula
	in the bounded domain $B_1\times(0,\infty)$ gives:
	\begin{equation*}
		I^A(x,t)=\min_{|y|\leq 1}\Big\{t L\Big(\frac{y-x}{t}\Big)+A\Big\}
		\wedge \min_{(y,s)\in \partial B_1\times(0,t)}
		\Big\{(t-s)L\Big(\frac{x-y}{t-s}\Big)\Big\}\,.
	\end{equation*}
	Using that $L\geq0$ and $L(0)=0$,
	we obtain that the first min equals $A$. Then using that the function $r\mapsto L(cr)/r$
	is increasing, the second minimum is attained for $s=0$ so that the result holds.
\end{proof}

\section{Theoretical Behaviour}\label{sect:theoretical}
\setcounter{equation}{0}

The main goal of this section is to derive a theoretical bound, in terms of
the Lagrangian $L$, for the error made when approximating the solution, $u$, of~\eqref{eq:0} by solutions, $u_R$,
of the Dirichlet problem. This result will allow us to derive explicit rates of convergence provided
we know the behaviour at infinity of $L(q)$ (this will be the aim of the next sections):

\begin{theorem}\label{thm:est-IR}
	If $J$ satisfies \eqref{eq:hyp:J} and \eqref{eq:hyp2:J},
	then for any fixed $x\in\R^N$ and $t>0$, as $R\to\infty$, there holds
	\begin{equation}
    \label{eq:est:sup:theo}
	|u-u_R|(x,t)\leq\e^{-R I_\infty(x/R,t/R)+o(1)R},
	\end{equation}
	where the rate function is given by
	\begin{equation}\label{est:IR}
		I_\infty(x,t)=\min_{y\in\partial B_1}t\,L\Big(\frac{x-y}{t}\Big)\,.
	\end{equation}
	Moreover, for any $T>0$
	the $o(1)$ is uniform in sets of the form
	\begin{equation*}
		\big\{|x/R|\leq 1\,,\ 0\leq t/R\leq T\big\}.
	\end{equation*}
\end{theorem}

We shall dedicate the rest of this section to the proof of this theorem. Notice that if one only assumes
\eqref{eq:hyp3:J} instead of \eqref{eq:hyp2:J}, the theorem still remains valid, but the proof requires
some finer arguments since the Hamiltonian can be infinite in some regions of $\R^N$. In such a case,
an initial layer appears that we analyse in detail in Section \ref{sect:critical}.

\subsection{Formal convergence}\label{subsect:formal}

Let us denote by $v_R$ the solution of \eqref{eq:0} in $B_R\times(0,\infty)$ with
initial value $v_R(x,0)=0$ for $x\in B_R$ and ``boundary data'' $v_R=\|u_0\|_\infty$ for $|x|\geq R$.
By comparison, $0\leq u-u_R\leq v_R$ so that we need only to estimate $v_R$.

We first rescale the equation both in $x$ and $t$ as follows:
\begin{equation*}
	w_R(x,t)=v_R(Rx, Rt)\quad\text{for}\quad x\in B_1,\ t\geq0\,.
\end{equation*}
Then $w_R$ satisfies a rescaled equation in the fixed ball $B_1$ in the sense of viscosity:
\begin{eqnarray*}
	\partial_t w_R(x,t)&=&\frac{1}{R}\tr\Big( a(Rx)D^2w_R(x,t)\Big) + b(Rx)\cdot D w_R(x,t) \\
&+& R\int_{\R^N}\Big\{w_R(x+y/R,t) - w_R(x,t)-(Dw_R(x)\cdot y/R) \ind{|y|<1}(y)\Big\}J(y)\dy.
\end{eqnarray*}

In order to  estimate $w_R$ we follow~\cite{BarlesDaherRomano94} and perform the ``usual'' logarithmic
transform, but we have to rescale accordingly, dividing by $R$ (and not $R^2$ as it is the case
for the heat equation).	So, remembering that for $t>0$, $w_R>0$, let us define
\begin{equation*}
	I_R(x,t)=-\frac{1}{R}\ln (w_R(x,t))\,.
\end{equation*}
Then
\[
\begin{aligned}
    &\partial_t w_R(x,t)=-R\e^{-RI_R(x,t)}\partial_tI_R(x,t)\,,\qquad Dw_R(x,t)=-R\e^{-RI_R(x,t)}DI_R(x,t)\,,\\[2mm]
	&D^2w_R(x,t)=-R\e^{-RI_R(x,t)}D^2I_R(x,t)+R^2\e^{-RI_R(x,t)}(DI_R\otimes DI_R)(x,t)\,,
\end{aligned}
\]
and
\[
\begin{aligned}
    \int \Big\{& w_R(x+y/R,t)-w_R(x,t)-(Dw_R(x)\cdot y/R) \ind{|y|<1}(y)\Big\}J(y)\dy\\
	&=\int \Big\{ \e^{-RI_R(x+y/R,t)}-\e^{-RI_R(x,t)}+(DI_R(x)\cdot y)\e^{-RI_R(x,t)} \ind{|y|<1}(y)\Big\}J(y)\dy\\
               &=\e^{-RI_R(x,t)}\int_{\R^N}\Big\{\e^{R\{-I_R(x+y/R,t)+I_R(x,t)\}}-1+(DI_R(x)\cdot y)
	\ind{|y|<1}(y)\Big\}J(y)\dy\,.
\end{aligned}
\]
As for the differential terms, we have
\[
\begin{aligned}
-&\e^{RI_R(x)}\Big(\frac{1}{R^2}\tr\Big( a(Rx)D^2w_R(x,t)\Big) + \frac{1}{R}b(Rx)\cdot D w_R(x,t)\Big)\\
&=\frac{1}{R}\tr\Big( a(Rx)D^2I_R(x,t)\Big) -\tr\Big( a(Rx)DI_R(x,t)\otimes DI_R(x,t)\Big)
	- b(Rx)\cdot D I_R(x,t)\,.
\end{aligned}
\]
We thus arrive at the following equation for $I_R$ (to be interpreted in the sense of viscosity, with test functions):
\begin{equation}\label{eq:IR}
    \begin{aligned}
	&\partial_t I_R(x,t)+\int_{\R^N}\Big\{\e^{-R\big\{I_R(x+y/R,t)-I_R(x,t)\big\}}-1+(DI_R(x)\cdot y)
	\ind{|y|<1}(y)\Big\}J(y)\dy\\
	&=\frac{1}{R}\tr\Big( a(Rx)D^2I_R(x,t)\Big) -\tr\Big( a(Rx)DI_R(x,t)\otimes DI_R(x,t)\Big)
	- b(Rx)\cdot D I_R(x,t)
    \end{aligned}
\end{equation}
Using the limits defined in \eqref{eq:limit:differentials},
the equation formally converges to the Hamilton-Jacobi equation $\partial_t I +H(DI)=0$ with
\begin{equation}\label{eq:H}
	H(p)=\tr\Big( A\, p\otimes p\Big)+B\cdot p+\int \Big(\e^{\,p\cdot y}-1-(p\cdot y)\ind{|y|<1}\Big)J(y)\dy\,.
\end{equation}
In the following subsection we then justify the convergence of $I_R$ towards the solution
in the sense of viscosity solutions.

\subsection{Passing to the limit in the viscous sense}\label{subsect:limit}

A first problem comes from the fact that if $w_{R}$ approaches zero, then $I_{R}$ may not
remain bounded. Hence to avoid upper estimates for $I_R$, we use the same trick as in \cite{BarlesDaherRomano94}
which consists in modifying $I_R$ a little bit. For any $A>0$, let
\begin{equation*}
	I_R^A(x,t)=-\frac{1}{R}\ln (w_R(x,t) + \e^{-RA}),
\end{equation*}
which is bounded from above by $A$. Let us notice that since equation~\eqref{eq:0} is invariant under addition of
constants, $I_R^A$ satisfies the same equation as $I_R$, that is, equation \eqref{eq:IR}.

\begin{proposition}
	The sequence $(I_R^A)$ converges locally uniformly in $\overline B_1\times[0,\infty)$ as $R\to+\infty$ towards the unique
	viscosity solution $I^A$ of \eqref{pb:limitA}.
\end{proposition}

\begin{proof}
 	We introduce the half-relaxed limits, for $x\in \overline{B_1},\,t\geq0$:
 	\begin{equation*}
 		\overline{I}^A(x,t) := \limsup_{R\to\infty}{}^*I^A_R(x,t)=
		\limsup_{\genfrac{}{}{0pt}{}{(x',t')\to(x,t)}{R\to\infty}}I^A_R(x',t')
 	\end{equation*}
 	and
     \begin{equation*}
         \underline{I}^A(x,t) := \liminf_{R\to\infty}{}^*I^A_R(x,t)=
		 \liminf_{\genfrac{}{}{0pt}{}{(x',t')\to(x,t)}{R\to\infty}}I^A_R(x',t'),
     \end{equation*}
 	and we shall prove that they are respectively viscosity sub- and super-solutions of the
 	limit problem \eqref{pb:limitA}. Then a uniqueness result will allow us to conclude.

	Let us take $\delta\in(0,1)$ and a test function $\phi$ such that $\overline{I}^A-\phi$ has a
	maximum at $(x_0,t_0)$.	Up to a standard modification of $\phi$, we can assume the maximum is strict so that
	there exist sequences $R_n\to+\infty$ and $(x_n,t_n)\to(x_0,t_0)$ such that
	\begin{equation*}
	I_{R_n}^A-\phi\text{ has a strict maximum at }(x_n,t_n)\,.
	\end{equation*}

	\textsc{Case 1: } the point $(x_0,t_0)$ is inside $B_1\times(0,\infty)$.
	Then for $n$ big enough, all the points $(x_n,t_n)$ are also inside $B_1\times(0,\infty)$
	so we may use the equation for $I_{R_n}^A$ at those points and pass to the limit.

	We first write down the viscosity inequality for $I_R^A$:
\[\begin{aligned}
		\frac{\partial \phi}{\partial t}&(x_n,t_n)\leq
	\frac{1}{R_n}\tr\Big( a(R_nx_n)D^2\phi(x_n,t_n)\Big)\\
& -\tr\Big( a(R_nx_n)D\phi(x_n,t_n)\otimes D\phi(x_n,t_n)\Big)
	- b(R_nx_n)\cdot D\phi(x_n,t_n) \\
		&-\int_{\{|y|<\delta\}}   \Big\{\e^{-R\{\phi(x_n+y/R,t_n)-\phi(x_n,t_n)\}}-1-
		(D \phi(x_n,t_n)\cdot y)\Big\}J(y)\dy\\
		&-\int_{\{|y|\geq\delta\}} \Big\{\e^{-R\{I_{R_n}^A(x_n+y/R,t_n)-I_{R_n}^A(x_n,t_n)\}}-1-
			(D \phi(x_n,t_n)\cdot y)\ind{|y|<1}\Big\}J(y)\dy\\[2mm]
	&={\rm Diff}(n) + {\rm Int}_1(n,\delta) + {\rm Int}_2(n,\delta)
	\end{aligned}\]
	where ${\rm Diff}(n)$ represents the differential terms, ${\rm Int}_1(n,\delta)$ is the integral over
	$\{|y|<\delta\}$ and ${\rm Int}_2(n,\delta)$ is the integral over $\{|y|\geq\delta\}$.

	Let us first remark that passing to the limit in the differential terms is easy,
	using \eqref{eq:limit:differentials}:
	\[
	{\rm Diff}(n)\to -\tr\Big( A\cdot D\phi(x_0,t_0)\otimes D\phi(x_0,t_0)\Big)- B\cdot D\phi(x_0,t_0)\,.
	\]

	For $n$ big enough, the first integral term can be controled by:
	\[
		|{\rm Int}_1(\delta,n)|\leq \|D^2\phi\|_{L^\infty(B_1(x_0,t_0))}
	\int_{\{|y|<\delta\}}|y|^2 J(y)\dy\to0\text{ as }\delta\to0\,.
	\]

	It remains to pass to the limit in ${\rm Int}_2$. To this end we use the fact that, since we have
	a maximum point, for any $z\in\R^N$,
	\begin{equation*}
		I_{R_n}^A(x_n+z,t_n)-I_{R_n}^A(x_n,t_n)\leq \phi(x_n+z,t_n)-\phi(x_n,t_n)\,.
	\end{equation*}
	Then, we fix $\eps>0$, choose some $M>1$ and split ${\rm Int}_2$ into two terms as follows:
	$$
    \begin{aligned}
		{\rm Int}_2(\delta,n)\leq &-\!\! \int_{\delta\leq|y|<M}\!\!\Big\{\!  \e^{-R\{\phi(x_n+y/R_n,t_n)
		-\phi(x_n,t_n)\}}-1-(D \phi(x_n,t_n)\cdot y)\ind{|y|<1}\Big\}	J(y)\dy \\
		 &+ \Big|\int_{|y|\geq M}\Big\{w_{R_n}(x_n+y/R_n,t_n)-w_{R_n}(x_n,t_n)
		\Big\}J(y)\dy\,\Big|\,.
	\end{aligned}
    $$
	Since $w_R$ is bounded by $\|u_0\|_\infty$, we can choose $M$ big enough so that the second term is less than $\eps$,
	independently of $n$.
	
	Then we write a Taylor expansion for $\phi$ near point $x_n$:
	there exists a $\xi_n\in B_M$ such that
	\[
		{\rm Int}_2\leq -\int_{\delta\leq|y|<M}\!\!\Big\{\e^{-D \phi(x_n,t_n)\cdot y+\frac{1}{R_n}
		(D^2\phi(\xi_n)y\cdot y)}-1-(D \phi(x_n,t_n)\cdot y)\ind{|y|<1}\Big\}J(y)\dy +\eps\,.
	\]
	Since $\xi_n$ remains in $B_M$ and  $\phi$ is smooth we have that $D^2\phi(\xi_n)$ remains bounded. Hence, we can pass to the limit as $n\to+\infty$:
	\[
		\limsup_{n\to\infty}{\rm Int}_2(\delta,n)\leq -\int_{\delta\leq|y|<M}
	\Big\{\e^{-D \phi(x_0,t_0)\cdot y}-1-(D \phi(x_0,t_0)\cdot y)\ind{|y|<1}\Big\}J(y)\dy +\eps\,.
	\]
	
	Summing up the various terms, we obtain that for any $\delta>0$ and any $\eps>0$, there exists $M=M(\eps)>1$ such that
\[
    \begin{aligned}
		\frac{\partial \phi}{\partial t}(x_0,t_0)\leq&
	-\tr\Big( A\cdot D\phi(x_0,t_0)\otimes D\phi(x_0,t_0)\Big)
	- B\cdot D\phi(x_0,t_0) \\
		&-\int_{\delta\leq|y|<M}
	\Big\{\e^{-D \phi(x_0,t_0)\cdot y}-1-(D \phi(x_0,t_0)\cdot y)\ind{|y|<1}\Big\}J(y)\dy +\eps+o_\delta(1)
	\end{aligned}
\]
	where $o_\delta(1)$ represent a quantity that goes to zero as $\delta\to0$.

	It only remains to pass to the limit as $\eps,\delta\to0$. Since $J$ satisfies \eqref{eq:hyp:J} and
	\eqref{eq:hyp2:J}, the integral over $\R^N$ converges, and we can
	send $M$ to $+\infty$ and obtain in the limit:
\[
    \begin{aligned}
		\frac{\partial \phi}{\partial t}(x_0,t_0)\leq&
	-\tr\Big( A\cdot D\phi(x_0,t_0)\otimes D\phi(x_0,t_0)\Big)
	- B\cdot D\phi(x_0,t_0) \\
		&-\int_{\R^N}
	\Big\{\e^{-D \phi(x_0,t_0)\cdot y}-1-(D \phi(x_0,t_0)\cdot y)\ind{|y|<1}\Big\}J(y)\dy\,,
	\end{aligned}
\]
	which shows that $\overline{I}^A$ at $(x_0,t_0)$ is a subsolution in the sense of viscosity.

\textsc{Case 2: } the point $(x_0,t_0)$ is located at the boundary, $x_0\in\partial B_1,t_0>0$.
	Then the sequence $(x_n,t_n)$ may have points $x_n$ either inside $B_1\times(0,\infty)$, or at the boundary,
	or even outside $B_1$.
	If $x_n\in B_1$, we use the equation as in the previous case while if
	$x_n\in\partial B_1$, we use the relaxed boundary condition in Definition \ref{def:relaxed:boundary}.
	Finally, if $|x_n|>1$, then $I_{R_n}(x_n,t_n)=0$
	so that in any case, one has
	\begin{equation*}
		\min\big\{\partial_t\phi+H(D\phi)\,;\,I_{R_n}^A\big\}
		\leq0\text{ at }(x_n,t_n)\,.
	\end{equation*}
 	We then pass to the limit as $n\to+\infty$ and get the relaxed condition for $\overline{I}^A$
 	at the boundary.

	\textsc{Case 3: } the point $(x_0,t_0)$ is located at $t_0=0$, $x_0\in B_1$.
	The same as in case 2 happens: we have either $t_n=0$ and then we use the initial condition,
	or $t_n>0$ in which case we use the equation. In any case we get
	\begin{equation*}
		\min\big\{\partial_t\phi+H(D\phi)\,;\,I_{R_n}^A-A\big\}
		\leq0\text{ at }(x_n,t_n)\,,
	\end{equation*}
	which gives in the limit
	\begin{equation*}
		\min\big\{\partial_t\phi+H(D\phi)\,;\,\overline{I}^A-A\big\}
		\leq0\text{ at }(x_0,t_0)\,.
	\end{equation*}
	Now, it is well-known (see for instance \cite[Thm 4.7]{BarlesBook})
	that in this case, the initial condition is equivalent
	to $\lim_{t\to0}\overline{I}^A\leq A$. Actually this can be proved as in Proposition \ref{prop:HJ.initial.trace},
	using that in the present situation, the Hamiltonian is finite everywhere.

	\textsc{Conclusion: }
	First, the supersolution conditions for $\underline{I}^A$ are obtained by the same
 	method, with reversed inequalities. Then using comparison between u.s.c./l.s.c. sub/super solutions for
	\eqref{pb:limitA}, we get the inequality $\underline{I}^A\leq\overline{I}^A$,
	which implies equality of both functions. Hence, all the sequence
	converges uniformly in $\overline B_1\times[0,T]$ for all $T>0$ to the unique solution $I^A$.~
\end{proof}

\subsection{Proof of Theorem \ref{thm:est-IR}}
	This result only comes from the fact that for any $A>0$, by construction
	\begin{equation*}
		\overline{I}^A=\inf(\overline{I},A),\quad \underline{I}^A=\inf(\underline{I},A),
	\end{equation*}
with
    \begin{equation*}
    \overline{I}(x,t) := \limsup_{R\to\infty}{}^*I_R(x,t),\qquad\underline{I}(x,t) := \liminf_{R\to\infty}{}^*I_R(x,t).
    \end{equation*}
The fact that $\overline{I}^A=\underline{I}^A$, together with Lemma \ref{lem:representation}
yields the result for fixed $(x,t)$, passing to the limit as $A\to\infty$.

Now, the convergence of $I_R$ to $I_\infty$ is locally uniform in $\overline B_1\times[0,\infty)$ so that for any $T>0$
as long as $x/R\leq1$ and $0\leq t/R\leq T$,
$$|I_R(x/R,t/R)-I_\infty(x/R,t/R)|=o(1)\to 0\quad\text{as}\quad R\to\infty\,,$$
where $o(1)$ is uniform with respect to $x$ and $t$ as above.
Thus estimate \eqref{eq:est:sup:theo} indeed holds, which ends the proof. Notice that
at $t=0$, both $I_R$ and $I_\infty$ are infinite (which corresponds to $A=+\infty$), but
anyway, the difference $I_R-I_\infty$ remains uniformly controlled.

\begin{remark}
  {\rm At this stage we have an estimate valid up to the boundary of $B_R$. In the next sections we shall  derive more explicit estimates only for $|x|\leq \theta R$, with $\theta\in(0,1)$, because we use the asymptotic behaviour of $L(q)$ as $|q|\to\infty$.}
\end{remark}

\section{Compactly supported kernels}\label{sect:compact}
\setcounter{equation}{0}

In this section we prove that a general  ``$R\ln R$'' bound is valid for compactly supported kernels, extending the
symmetric and regular case proved in \cite{BrandleChasseigne08-1}. In order to take into account the possible
asymmetry of the kernel, we define below for any unit vector $\mathbf{\nu}$, the size of the support of $J$
in the direction $\mathbf{\nu}$:
\begin{definition}
	For any $\mathbf{\nu}\in\R^N$ with $|\mathbf{\nu}|=1$, let
	\begin{equation}\label{eq:compact.support.rho}
	\rho(\mathbf{\nu}):=\sup\big\{ r>0: J(r\mathbf{\nu})>0 \big\}\,.
	\end{equation}
\end{definition}
Notice that since $J$ is continuous, for any $r>0$ close enough to $\rho(\mathbf{\nu})$, $J$ is positive in a neighborhood
of $r\mathbf{\nu}$. Notice also that if $J$ is symmetric, then $\rho(\mathbf{\nu})=\rho$, the radius of the support of $J$.
We shall first derive a bound from below for non-symmetric kernels in the logarithmic scale:
\begin{lem}
  	\label{lem:compactsupport}
	Let $J$ be a continuous compactly supported kernel, let $\mathbf{\nu}=p/|p|$ and define $\rho(\mathbf{\nu})$
	as above. Then we have:
  	\begin{equation}
        \label{eq:behaviour:DpHp:compact}
        \liminf_{|p|\to\infty}\frac{\ln(p\cdot D\Hess(p))}{\rho(\mathbf{\nu}) |p|}\geq1\,.
    \end{equation}
\end{lem}

\begin{proof}
Let us first choose a unit vector $\mathbf{\nu}$, define $\rho(\mathbf{\nu})$ by \eqref{eq:compact.support.rho}
and consider $p\in\R^N$ going to infinity in
this direction: $p/|p|=\mathbf{\nu}$ and $|p|\to\infty$. We begin with writing
\[
    p\cdot D\Hess(p)=
    \int_{\{p\cdot y\leq 0\}\cap\{|y|>\rho_0/2\}} p\cdot y\e^{p\cdot y}J(y)\dy+\int_{\{p\cdot y> 0\}\cap\{|y|>\rho_0/2\}} p\cdot y\e^{p\cdot y}J(y)\dy
\]
The first term is bounded by from below by $-c|p|$ and for the second one we define
for $\varepsilon>0$ and $\beta<1$, the set
    \begin{equation*}
    \mathcal{C}^+_{\varepsilon,\beta}=\{y: \frac{\rho(\mathbf{\nu})}{\beta}\leq |y|\leq \rho(\mathbf{\nu}),\
	p\cdot y\geq (1-\varepsilon)|p||y|\geq 0\}\cap\{|y|>\rho_0/2\}\,.
    \end{equation*}
 Hence
    \begin{equation*}
    \begin{aligned}
    \int_{\{p\cdot y> 0\}\cap\{|y|>\rho_0/2\}} p\cdot y\e^{p\cdot y}J(y)\dy
            &\geq |p|\frac{\rho(\mathbf{\nu})}{\beta}(1-\varepsilon)\e^{(1-\varepsilon)|p|\frac{\rho(\mathbf{\nu})}{\beta}}
            \int_{\mathcal{C}^+_{\varepsilon,\beta}}J(y)\dy\\
            & \geq
            C(\varepsilon,\beta)|p|\frac{\rho(\mathbf{\nu})}{\beta}(1-\varepsilon)\e^{(1-\varepsilon)|p|
			\frac{\rho(\mathbf{\nu})}{\beta}}.
    \end{aligned}
    \end{equation*}

Notice that $C(\eps,\beta)>0$
 since $J$ is continuous and positive near $\rho(\mathbf{\nu})\mathbf{\nu}$, even if this constant
could be small. Summing up,
    \begin{equation*}
        p\cdot D\Hess(p)\geq  C(\varepsilon,\beta)|p|\frac{\rho(\mathbf{\nu})}{\beta}(1-\varepsilon)\e^{(1-\varepsilon)|p|\frac{\rho(\mathbf{\nu})}{\beta}}-|p|c
            \geq KC(\varepsilon,\beta)|p|\frac{\rho(\mathbf{\nu})}{\beta}(1-\varepsilon)\e^{(1-\varepsilon)|p|\frac{\rho(\mathbf{\nu})}{\beta}},
    \end{equation*}
for some constant $K$. Therefore, we obtain for every $\beta$ and $\varepsilon$
    \begin{equation*}
        \liminf_{|p|\to\infty}\frac{\ln(p\cdot D\Hess(p))}{\rho(\mathbf{\nu}) |p|}\geq
        \liminf_{|p|\to\infty}\left(\frac{\ln C(\varepsilon,\beta)}{\rho(\mathbf{\nu}) |p|}+\frac{\ln(\frac{\rho(\mathbf{\nu})}{\beta}|p|(1-\varepsilon))}{\rho(\mathbf{\nu}) |p|}+\frac{\frac{\rho(\mathbf{\nu})}{\beta} |p|(1-\varepsilon)}{\rho(\mathbf{\nu}) |p|}\right)=\frac{1-\varepsilon}{\beta}.
    \end{equation*}
Now, letting $\varepsilon\to 0$ and $\beta\to 1$ we conclude that~\eqref{eq:behaviour:DpHp:compact} holds.
    \end{proof}

In other words, using Lemma~\ref{lemma:H:less:pDHp} and~\ref{lemma:H:singular}, we have obtained
       $$ \liminf_{|q|\to\infty}\frac{\ln L(q)}{\rho(\mathbf{\nu}) |p_0(q)|}\geq1\,.$$

Then in order to have a more explicit bound using Theorem \ref{thm:est-IR}, we shall compare with a symmetric kernel,
using then the radius of the support of $J$.

\begin{theorem}\label{theo:compact:support}
	Let $J$ be a compactly supported kernel satisfying \eqref{eq:hyp:J}. We denote by $\rho$ the
	size of the support of $J$: $$\rho=\inf\{r>0:\supp(J)\subset B_r\}\,.$$
	Then the following estimate holds: for any $\theta\in(0,1)$ and $T>0$, as $R\to\infty$,
	\begin{equation}
        \sup_{|x|\leq\theta R\atop 0\leq t\leq TR}|u-u_R|(x,t)\leq \e^{-\frac{(1-\theta)}{\rho}R\ln R+o(R\ln R)}
    \end{equation}
\end{theorem}

Notice that	$J$ can be asymmetric and have a singularity at the origin.

\begin{proof}
	We first use Lemma \ref{lemma:reduce:symmetric} to reduce our estimate to the case
	of symmetric, compactly supported kernels. More precisely, using Lemma \ref{lemma:reduce:symmetric},
	if $J_*$ is a symmetric	kernel such that $J\curlyeqprec J_*$ and $\supp(J_*)=B_\rho$, then for $|p|$ big enough,
	$L_*\leq L$ where $L_*$	is the Lagrangian associated to $J_*$.
	Taking a look at Theorem \ref{thm:est-IR}, this implies that
	\begin{eqnarray*}
		-RI_\infty(x/R,t/R)&=&-R\min_{y\in \partial B_1}(t/R)L\Big(\frac{x/R-y}{t/R}\Big)\\
		&\leq&-R\min_{y\in\partial B_1}(t/R)L_*\Big(\frac{x/R-y}{t/R}\Big)\\
		&\leq&-\min_{y\in\partial B_R}t\,L_*\Big(\frac{x-y}{t}\Big)\,.\\
	\end{eqnarray*}
	Now we assume that $|x|<\theta R$ so that in this set $|x-y|\geq(1-\theta)R\to\infty$ and we shall use
	the behaviour at infinity of $L_*$. Lemma \ref{lemma:H:singular} allows us to wipe out the possible
	singular part near the origin as well as the differential terms of the Hamiltonian.

	Since $J_*$	is symmetric, so is $\Less_*$ so that, noting $\Less_*(x)=\underline{L}_*(|x|)$ we get
	\begin{equation*}
		\liminf_{R\to\infty}\frac{-RI_\infty(x/R,t/R)}{-\min\limits_{y\in\partial B_R}t\underline{L}_*\Big(\dfrac{|x-y|}{t}\Big)}\leq 1\,.
	\end{equation*}
	
	Then we use the results of \cite[Lemma 4.1 and Corollary 4.2]{BrandleChasseigne08-1} applied to
	$\underline{L}_*$ which is symmetric, associated to a nonsingular kernel to conclude.
	\end{proof}

\section{Intermediate kernels}\label{sect:intermediate}
\setcounter{equation}{0}

We consider now a general kernel $J$ satisfying \eqref{eq:hyp:J},
positive everywhere in $\R^N$, so that we can always write
    $$
        J(y)=\e^{-|y|\omega(y)}\,, \text{ with }\omega(y)=-\frac{\ln J(y)}{|y|}\,.
    $$
We will now make some further assumptions in this section:
    \begin{equation}
    \label{eq:hyp:yw}\begin{cases}
		J\text{ is }\mathrm{C}^1\text{-smooth for }|y|>0,\\[2mm]
		|DJ(y)|\text{ is bounded on }\{|y|>\eps\},\ \forall\eps>0,\\[2mm]
        y\mapsto|y|\omega(y) \mbox{ is superlinear and convex,}\\[2mm]
	\ds\exists\eta\in(0,1]\,,\ \liminf_{|y|\to+\infty}
	\frac{( y\cdot D\omega(y))}{|y||D\omega(y)|}\geq\eta.
	\end{cases}
    \end{equation}

Let us comment these hypotheses:
\begin{enumerate}
\item The regularity assumption on $J$ (which implies the same regularity for $\omega$) is not crucial since
by comparison we can deal with less regular kernels, using the results of Section~\ref{sect:preliminaries}.
\item About the convexity of $|y|\omega(y)$, it is actually only required for large $|y|$ for the same reason:
we only care about the tail of $J$.
\item The case of compactly supported kernels, which would correspond to $\omega(y)=+\infty$
outside a ball is treated in Section \ref{sect:compact}. On the other hand, the case of critical kernels
treated in Section \ref{sect:critical} corresponds to $\lim \omega(y)=\ell<\infty$.
So, the assumption of superlinearity is in between: $\lim \omega(y)=\infty$. This is why we speak of
\textit{intermediate} kernels here.
\item The superlinearity assumption implies that $J$ automatically satisfies \eqref{eq:hyp2:J} since indeed,
	for any $\beta>0$, we have that $\omega(y)>\beta$ for $|y|$ large enough.
\item The convexity and superlinearity assumptions altogether allow us to define the Legendre transform
	$\mathcal{K}(\cdot)$ of $|y|\omega(y)$, which will also be superlinear and convex, see~\cite{Rockafellar}:
    \begin{equation}
    \label{eq:K}
        \mathcal{K}(p):=\sup_{y\in\R^N}\big\{p\cdot y -|y|\omega(y)\big\}\,.
    \end{equation}
This function $\mathcal{K}(\cdot)$ will play a big role in estimating the rate of convergence. Notice that it is also the Legendre transform of $\ln(1/J)$.
\item The ``angle'' condition on $D\omega$ says that the gradient cannot take a purely tangential position. This is
a very weak assumption in this form that allows us to derive a minimum behaviour for non-symmetric kernels.
\end{enumerate}

\subsection{Properties of $\mathcal{K}$}

Thanks to Lemma \ref{lemma:H:less:pDHp}, we know that $$L(q)\sim p_0(q)\cdot D H(p_0(q)),$$
where $p_0(q)$ is such that $L(q)=p_0\cdot q -H(p_0)$. Thus, a main step consists in
finding a lower bound for $p\cdot DH(p)$. Here is where  $\mathcal{K}(p)$ plays an important role: roughly speaking, we will see that
$$
\lq\lq\, p\cdot DH(p)=\int_{\R^N}p\cdot y \e^{p\cdot y-|y|\omega(y)}\dy \geq \e^{\mathcal{K}(p)}\,"\,.
$$
Hence a detailed study of the properties of $\mathcal{K}$ is needed. To this aim, let $y_0(p)$ be the point where the sup in
\eqref{eq:K} is attained.

\begin{lem}
    \label{lemma:maximum}
	Let us assume that $J(y)=\e^{-|y|\omega(y)}$ satisfies \eqref{eq:hyp:J} and \eqref{eq:hyp:yw}.
	Then the function $\mathcal{K}$ is nonnegative and $\mathcal{K}(0)=0$.
	Moreover, for $|p|$ large enough, the supremum in~\eqref{eq:K} is attained at a unique point $y_0=y_0(p)$ and
	$|y_0(p)|\to\infty$ as $|p|\to\infty$.
\end{lem}

\begin{proof}
For $p\in\R^N$ fixed, let us denote by $$\varphi_p(y):=p\cdot y-|y|\omega(y)\,.$$
Then observe that $\varphi(0)=0$, so that $\mathcal{K}(p)=\sup \varphi_p\geq0$ and
of course $\mathcal{K}(0)=0$.
The superlinearity $y\mapsto|y|\omega(y)$ implies
that $\varphi_p(y)\to-\infty$ as $|y|\to\infty$, so that there exists at least
a maximum point $y_0(p)$. Now, any maximum point satisfies:
$$p=D(|y|\omega(y))(y_0)=\frac{DJ(y_0)}{J(y_0)}\,,$$
and by our assumptions, $J(y)\to0$ as $|y|\to\infty$ while $|DJ|$ remains bounded
away from $y=0$. Thus, as $|p|\to\infty$, necessarily $|y_0(p)|\to+\infty$.

We shall now prove that for $|p|$ large enough, the maximum point $y_0(p)$ is unique.
Since $y\mapsto|y|\omega(y)$ is convex for large $|y|$, say $|y|>C$, then $\varphi_p$ is concave for $|y|>C$,
independently of $p$. Indeed, this comes from the fact that $D^2\varphi_p(y)=-D^2(|y|\omega(y))$.
Thus we take $M>0$ large enough so that for any $|p|>M$, any minimum point $y_0(p)$ satisfies
$|y_0(p)|>C$ and enters the region where $\varphi_p$ is concave.
Then, the only case when there may exist several maximum points is the case when
$\varphi_p$ would be constant on some open set. But since $\omega(y)\to+\infty$ as $|y|\to\infty$,
this cannot happen for $y$ large. Hence if $|p|$ is large, the maximum is attained at a unique point
$y_0(p)$.
\end{proof}

%

\begin{lem}
	Under hypothesis \eqref{eq:hyp:yw}, we have
	\begin{equation}
		\liminf_{|p|\to+\infty}
		\frac{( p\cdot y_0(p))}{|p||y_0(p)|}\geq\eta\,.
	\end{equation}
\end{lem}

\begin{proof}
	By the definition of $y_0(p)$, we have $p=D(|y|\omega(y))(y_0)$ so that
	$$|p|\leq\omega(y_0)+|y_0||D\omega(y_0)|$$
	and
	$$ ( y_0\cdot p)=|y_0|\omega(y_0)+|y_0|( y_0 \cdot D\omega(y_0)).$$
	This implies, using  \eqref{eq:hyp:yw} that
	$$( y_0\cdot p)\geq\eta\big(|y_0|\omega(y_0)+|y_0|^2|D\omega(y_0)|\big)=\eta|y_0||p|\,,$$
	hence the result follows taking the liminf.
\end{proof}

\subsection{Estimating $L(q)$}

We begin by a technical lemma that will help us in constructing
a box where $p\cdot y-|y|\omega(y)$ is close to $\mathcal{K}(p)$ (see below):

\begin{lem}\label{lemma:unif.eta}
	There exists $M>0$ such that for any $|p|>M$ and
	any $\xi\in B(y_0(p),1/|p|)$, we have
	\[
		\frac{( D\omega(\xi)\cdot \xi)}{|D\omega(\xi)||\xi|}\geq\frac{3\eta}{4}>0\,,
	\]
with $\eta$ defined in~\eqref{eq:hyp:yw}.
\end{lem}

\begin{proof}
	Notice first that since $\xi$ is at distance at most $1/|p|$ of $y_0$ (which goes at infinity as $|p|\to\infty$),
	then all the points $\xi\in  B(y_0(p),1/|p|)$ verify that $|\xi|$ is big provided $|p|$ is also big.
	More precisely, using \eqref{eq:hyp:yw} we know there exists a $m>0$ such that for any $|y|>m$,
	\[
		\frac{( D\omega(y)\cdot y)}{|D\omega(y)||y|}\geq\frac{3\eta}{4}>0\,.
	\]
	Now, since $|y_0(p)|\to\infty$ as $|p|\to\infty$, there exists $M>0$ such that
	if $|p|>M$ then any	$|\xi|\in B(y_0(p),1/|p|)$ verifies $|\xi|>|y_0|-1/|p|>m$. Then
	we may apply the above estimate taking $y=\xi$, which gives the result.
\end{proof}

Let us now make clear some definitions:

\begin{definition}\label{def:cone}\label{def:angle}
	The angle between two vectors $a,b\in\R^N$ is defined as follows:
	\[
		\ang(a,b)=\arccos\frac{\pe{a}{b}}{|a||b|}\in[0,\pi]\,.
		\]
	Moreover, given a vector $a\in\R^N$ and an angle $\alpha\in[0,\pi/2]$, we define the
	positive cone $\mathcal{C}^+$ in the direction $a$ with aperture $\alpha$
	as follows:
	\[
		\mathcal{C}^+(a,\alpha)=\big\{ b\in\R^N:\ang(a,b)\leq\alpha \big\}\,.
	\]
	Accordingly we define the negative cone as follows:
	\[
		\mathcal{C}^-(a,\alpha)=\mathcal{C}^+(-a,\alpha)\,.
	\]
\end{definition}

Notice that by considering only apertures $0\leq\alpha\leq\pi/2$ (which is enough for our purpose here),
we make sure that $\mathcal{C}^-(a,\alpha)\cap\mathcal{C}^+(a,\alpha)=\{0\}$.
The we have the following lemma:

\begin{lem}
  \label{lemma:caja}
	Let us consider the cone
	$\mathcal{C}^*(p):=\mathcal{C}^-\big(y_0(p),\pi/2-\arccos(\eta/2)\big)$
	and $$\mathcal{A}(p)=\{y_0(p)+z : |z|\leq \frac{1}{|p|},\  z\in\mathcal{C}^*(p)\}\,.$$
	Then for $|p|$ big enough, we have
	$$\forall y\in\mathcal{A}(p)\,,\quad p\cdot y-|y|\omega(y)\geq p\cdot y_0-|y_0|\omega(y_0)-1\,.$$
	Moreover, the volume of $\mathcal{A}(p)$ is bounded by
	  $$
	  \mathcal{A}(p)|\geq\frac{c(N,\eta)}{|p|^N}
	  $$
	for some constant $c(N,\eta)>0$.
\end{lem}

\begin{proof}
	In all the proof we consider at least $|p|>M$ so that we may apply Lemma \ref{lemma:unif.eta} and
	take $y\in\mathcal{A}(p)$. Then $y\in B(y_0,\frac{1}{|p|})$ so that we have
	$\pe{p}{y}\geq\pe{p}{y_0}-1$.

	Moreover, since $0<\eta\leq1$,  the aperture of the cone $\mathcal{C}^*(p)$,
	$\arccos(\eta/2)$, is strictly positive and not equal to $\pi/2$. Hence for $|p|$ big enough,
	$\mathcal{A}(p)\subset B_{|y_0|}$ so that $|y|\leq |y_0|$ and it is enough to check that
	$\omega(y)\leq \omega(y_0)$ to get what we want, that is:
	$$\pe{p}{y}-|y|\omega(y)\geq \pe{p}{y_0}-|y_0|\omega(y_0)-1\,.$$

  	To this end, we write
  $$
  \omega(y)-\omega(y_0)=D\omega(\xi)\cdot(y-y_0)
  $$
  for some $\xi\in [y,y_0]$. The point is that, unless we are in a symmetric case,
  we do not have a very precise estimate for
  $D\omega(\xi)$. But we will prove that if $y\in\mathcal{C}^*(p)$ then necessarily
  \begin{equation}
    \label{eq:caja1}
      D\omega(\xi)\cdot(y-y_0)\leq 0\,,
  \end{equation}
	which will be enough to get the desired estimate.

\begin{figure}
  [ht]
  \begin{center}
    \includegraphics[width=6cm]{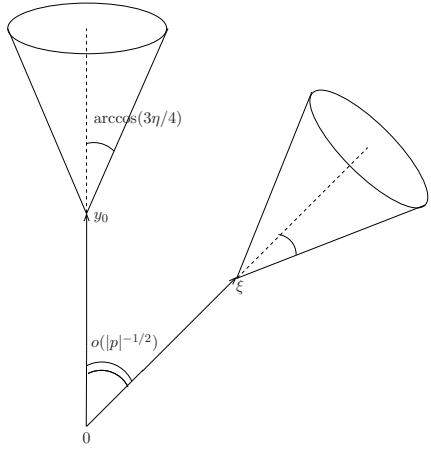}
    \caption{Cones in $y_0$ and $\xi$.}
    \label{fig:conos1}
  \end{center}
\end{figure}
    We shall first show that the image set $D\omega\big(B_{1/|p|}(y_0(p))\big)$ si contained in
	a cone of aperture comparable to $\arccos(3\eta/4)$, in the direction $y_0$. More precisely,
	we claim that for any $\xi\in B(y_0,1/|p|)$,
	\[
 		0\leq\ang (D\omega(\xi),y_0)\leq\arccos(3\eta/4)+o(1/|p|^{1/2})\,,
	\]
	see Figure~\ref{fig:conos1}. Indeed, we first estimate the angle $\ang(y_0,\xi)$ for any $\xi=y_0+z\in B(y_0,1/|p|)$ as follows:
	\begin{eqnarray*}
        \frac{\pe{y_0}{\xi}}{|y_0||\xi|}&=&\frac{\pe{y_0}{y_0}}{|y_0|(|y_0|+O(1/|p|))}+
		\frac{\pe{y_0}{z}}{|y_0|(|y_0|+O(1/|p|))}\\
        &=&1-o(1/|p|)
\end{eqnarray*}
	(recall that $|y_0(p)|\to\infty$), where the term $o(1/|p|)$ is nonnegative.
	The expansion,  as~$x\to0^-$, $\arccos(1+x)=O(x^{1/2})$   implies
	\[
		\ang(y_0,\xi)=o(1/|p|^{1/2})
	\]
	and finally, we use the fact that
  $$
 	\ang(D\omega(\xi),y)\leq \ang(D\omega(\xi),\xi)+\ang(\xi,y_0)
  \leq \arccos(3\eta/4)+o(1/|p|^{1/2})\,.
  $$

    Thus, for $|p|$ large enough (recall that the arccos function
	is strictly decreasing on $(0,1)$),
	$$
 	\ang(D\omega(\xi),y_0)\leq \arccos(\eta/2)\,.
  	$$
	Figure~\ref{fig:conos} shows the vectorial cone of aperture $\arccos(3\eta/4)$ in the direction $\xi$ which is included in the vectorial cone of aperture $\arccos(\eta/2)$ in the direction $y_0$. Hence, if we choose a point $y=y_0+z$ such that $z\in\mathcal{C}^-(y_0,\pi/2-\arccos(\eta/2))$ we make sure
	that $\ang(D\omega(\xi),z)\geq\pi/2$, which yields
	$$D\omega(\xi)\cdot z=D\omega(\xi)\cdot(y-y_0)\leq0\,.$$

\begin{figure}
  [ht]
  \begin{center}
    \includegraphics[width=8cm]{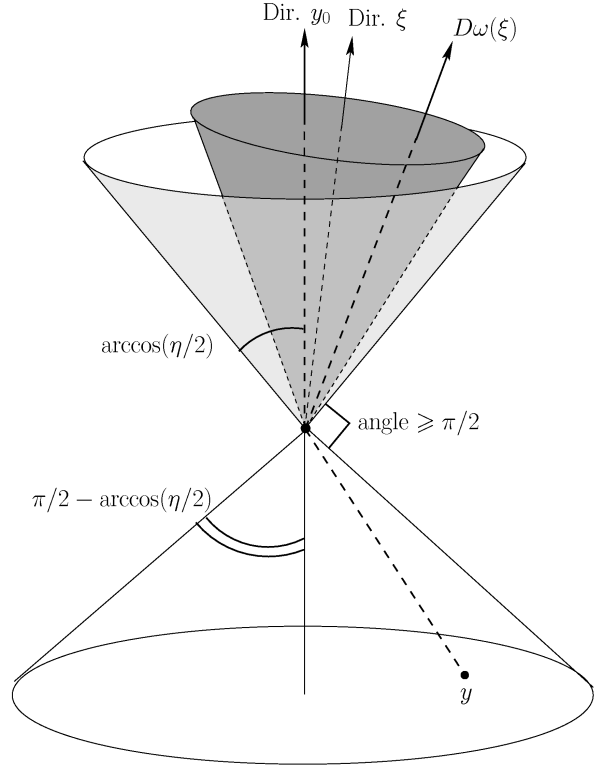}
    \caption{Position of the vectorial cones.}
    \label{fig:conos}
  \end{center}
\end{figure}

    To end the lemma, we only need to mention that $\mathcal{A}(p)$ is given by the intersection of the
	ball $B(y_0,1/|p|)$ with a cone placed at $y_0$, of aperture $\pi/2-\arccos(\eta/2)$. Hence its volume is
	indeed given $c(N,\eta)/|p|^N$ for some constant $c(N,\eta)>0$.
\end{proof}

    \begin{remark}{\rm
	As $\eta$ becomes close to zero, the aperture of the cone $\mathcal{C}^*(p)$ becomes very small:
	$\pi/2-\arccos(\eta/2)\sim 0$. So in the limit case $\eta=0$ we would not be able to
	construct a non-empty set $\mathcal{A}(p)$, at least with this method. Now, the size of
	$\mathcal{A}(p)$ gets small as $|p|\to\infty$ since it behaves like $|p|^{-N}$, but nevertheless
	this will be sufficient to get a suitable estimate, since inside this (small) region, we
	will get an exponential behaviour.
}
\end{remark}

\begin{lem}\label{lem:gen:est}
    Let $J(y)=\e^{-|y|w(y)}$ satisfying \eqref{eq:hyp:J} and \eqref{eq:hyp:yw}. Then following estimate holds:
    \begin{equation}\label{eq:gen:liminf}
        \liminf_{|p|\to\infty}\frac{\ln(p\cdot D\Hess(p))}{\mathcal{K}(p)}\geq 1.
    \end{equation}
\end{lem}

\begin{proof}
	We start from
	\begin{equation*}
		\begin{aligned}
        p\cdot D\Hess(p)&=\int_{\{|y|>\rho_0/2\}}p\cdot y\e^{p\cdot y-|y|w(y)}\dy\\
        &=
	   \int_{\{p\cdot y\leq 0\}\cap\{|y|>\rho_0/2\}}\big(\dots\big)\dy + 
		\int_{\{p\cdot y>0\}\cap\{|y|>\rho_0/2\}}\big(\dots\big)\dy\,, 
		\end{aligned}
    \end{equation*}
	and the integral over $\{p\cdot y\leq 0\}\cap\{|y|>\rho_0/2\}$ is bounded by $C|p|$ for some $C>0$,
	so that it will be negligible
	(recall that $p\cdot D\Hess(p)$ behaves at least exponentially -- see Lemma~\ref{lemma:H:basic}).
	Thus we shall give an estimate of $p\cdot D\Hess(p)$ only in terms of the integral over
	$\{p\cdot y>0\}\cap\{|y|>\rho_0/2\}$:
	for $|p|$ large enough, we have
	$$
	p\cdot D\Hess(p)\geq \frac{1}{2}\int_{\{p\cdot y>0\}\cap\{|y|>\rho_0/2\}} p\cdot y \e^{p\cdot y-|y|\omega(y)}\dy\\
	$$
	Now let us notice that, for $|p|$ big enough, $\mathcal{A}(p)\subset\{p\cdot y>0\}\cap\{|y|>\rho_0/2\}$.
	So we can write, using
	Lemma~\ref{lemma:caja} and \eqref{eq:hyp:yw}:
 	\begin{eqnarray*}
	p\cdot D\Hess(p)&\geq& \frac{1}{2}\int_{\mathcal{A}(p)} p\cdot y \e^{p\cdot y-|y|\omega(y)}\dy\\
	&\geq& \frac{1}{2}\eta|p|\int_{\mathcal{A}(p)}|y|\e^{\mathcal{K}(p)-1}\dy\\
	&\geq& \frac{1}{2}\eta|p||\mathcal{A}(p)|\Big(|y_0(p)|-\frac{1}{|p|}\Big)\e^{\mathcal{K}(p)-1}
	\end{eqnarray*}
	Hence, taking logarithms and dividing by $\mathcal{K}(p)$:
 	\begin{eqnarray*}
	\frac{\ln\big(p\cdot D\Hess(p)\big)}{\mathcal{K}(p)}&\geq& \frac{\ln(\eta|p|/2)}{\mathcal{K}(p)}+
		\frac{\ln|\mathcal{A}(p)|}{\mathcal{K}(p)}+\frac{\ln(|y_0(p)|-1/|p|)}{\mathcal{K}(p)}+1-\frac{1}{\mathcal{K}(p)}\,.
	\end{eqnarray*}
	Recall that $\mathcal{K}(p)$ is superlinear, $|y_0(p)|\to\infty$ and $|\mathcal{A}(p)|\geq c|p|^{-N}$ to conclude that
	$$
        \liminf_{|p|\to\infty}\frac{\ln(p\cdot D\Hess(p))}{\mathcal{K}(p)}\geq 1.
    $$
\end{proof}

We assume now that $J_*(y)=\e^{-|y|\omega_*(y)}$ is symmetric in order to get a more precise estimate. We denote by
$L_*$ and $\mathcal{K}_*$ the associated Lagrangian and Legendre transform of $|y|\omega_*(y)$.
Since in this case
$\mathcal{K}_*$ is symmetric (and still superlinear), we know that for $|p|$ large enough,
$\mathcal{K}_*^{-1}:\R_+\to\R^N$ is defined and for any $p\in\mathcal{K}_*^{-1}(|z|)$, we have
$|p|=$constant.

\begin{proposition}
	Let us assume that $J_*$ is a symmetric kernel satisfying \eqref{eq:hyp:J} and \eqref{eq:hyp:yw}. Then
	\begin{equation}\label{est:gen:L}
	\liminf_{|q|\to\infty}\frac{L_*(q)}{|q||\mathcal{K_*}^{-1}(\ln|q|)|}\geq1 \,.
	\end{equation}
\end{proposition}
\begin{proof}
	In this proof, we drop the $*$-subscript for simplicity.
	We use estimate \eqref{eq:gen:liminf} and, since $J$ is symmetric,
	$\Hess(\cdot)$ is also symmetric. Hence $p_0\cdot D\Hess(p_0)=|p_0||D\Hess(p_0)|$ so that,
	using that $\mathcal{K}(\cdot)$ is superlinear:
	$$\liminf_{|q|\to\infty}\frac{\ln |p_0|+\ln |D\Hess(p_0)|}{\mathcal{K}(p_0)}=
		\liminf_{|q|\to\infty}\frac{\ln |D\Hess(p_0)|}{\mathcal{K}(p_0)}\geq1\,.$$
	In other words, this means that for any $\eps\in(0,1)$, if $|q|$ is large enough,
	$$\ln|q|\geq (1-\eps)\mathcal{K}(p_0(q))\,.$$
	Using again that $\mathcal{K}$ is superlinear and symmetric, we have that for any $\eps\in(0,1)$, it is enough
	to take $|q|$ big to get
	$$ |p_0(q)|\leq|\mathcal{K}^{-1}((1-\eps)^{-1}\ln|q|)|\,.$$
	From this we get that for any $\eps\in(0,1)$:
	$$\liminf_{|q|\to\infty}\frac{p_0\cdot D\Hess(p_0)}{|q||\mathcal{K}^{-1}((1-\eps)^{-1}\ln|q|)|}\geq1,$$
	and we pass to the limit as $\eps\to0$ to get the result for $\Less(q)\sim p_0\cdot D\Hess(p_0)$:
	$$\liminf_{|q|\to\infty}\frac{\Less(q)}{|q||\mathcal{K}^{-1}(\ln|q|)|}\geq1.$$
	Finally, we invoke Lemma \ref{lemma:H:singular} to conclude that the result holds for $L(q)$.
\end{proof}

\subsection{Conclusion}

We are now ready to prove one of the main result of this paper:

\begin{theorem}
  \label{teo:main}
	Let $J(y)=\e^{-|y|w(y)}$ be a kernel satisfying \eqref{eq:hyp:J} and~\eqref{eq:hyp:yw}. Let us consider
	a symmetric kernel $J_*(y)=\e^{-|y|\omega_*(y)}$ such that $J\curlyeqprec J_*$, and denote by
	$\mathcal{K}_*$ the associated Legendre transform of $|y|\omega_*(y)$.
	Then the following estimate holds as $R\to\infty$: for any $\theta\in(0,1)$, $T>0$,
  	$$
  	\sup_{|x|\leq\theta R\atop 0\leq t\leq TR} |u-u_R|(x,t)\leq \e^{-(1-\theta)R\, \mathcal{K}_*^{-1}(\ln ((1-\theta)R/t)) (1+o(1))}.
  	$$
\end{theorem}

\begin{proof}The proof essentially follows the one in the case of compactly supported kernels:
	we first use Lemma \ref{lemma:reduce:symmetric} to reduce our estimate to the case
	of a symmetric kernel $J\curlyeqprec J_*$:
	\begin{eqnarray*}
		-RI_\infty(x/R,t/R)&=&-R\min_{y\in \partial B_1}(t/R)L\Big(\frac{x/R-y}{t/R}\Big)\\
		&\leq&-R\min_{y\in\partial B_1}(t/R)L_*\Big(\frac{x/R-y}{t/R}\Big)\\
		&\leq&-t\,\underline{L}_*\Big(\frac{\dist(x;\partial B_R)}{t}\Big)\,,\\
	\end{eqnarray*}
	where $\underline{L}_*(|x|)=L_*(x)$ is the symmetric Lagrangian associated to $J_*$.

	Now we assume that $|x|<\theta R$ so that in this set $\dist(x;\partial B_R)\geq(1-\theta)R\to\infty$
	and we use the behaviour at infinity of $L_*$ given by \eqref{est:gen:L} to get the result.
	
\end{proof}

\begin{corollary}In particular, if $x$ remains in a bounded set $B_M$, we can take any $\theta>0$ and if $t$ remains also
bounded we obtain a simpler estimate:
$$
	\sup_{|x|<M,0<t<1} |u-u_R|(x,t)\leq \e^{-R\, \mathcal{K}_*^{-1}(\ln R) (1+o(1))}.
$$
\end{corollary}

Several remarks are to be made:
\begin{enumerate}	

	\item The authors gave in \cite{BrandleChasseigne08-1} some explicit estimates, which consists here more or less
	in expliciting the function $\mathcal{K}(\cdot)$. We refer to Section \ref{subsect:explicit:estimates}
	for a list of known explicit behaviours.
	\item Even if we are able to prove a lower estimate for asymmetric kernels -- Lemma \eqref{lem:gen:est} --
	using $\mathcal{K}(p)$ which is not symmetric in general, we are facing a problem: knowing the behaviour of
	$p\cdot D H(p)$ is not enough to know the behaviour of each of the vectors, unless we make sure that
	they point more or less in the same direction. And this is not clear unless the kernel is ``almost'' symmetric
	because of the exponential behaviour. This is why we compare with the smallest symmetric kernel above $J$ in order
	to have a more explicit behaviour.
	\item Even if we were able to derive a bound taking into account the asymmetry,
	then we would have to study the min in Theorem \ref{thm:est-IR}, which is again not obvious unless we have
	an almost symmetric lagrangian.
	\item However, see Section \ref{subsect:asymetric:1D} for the 1-D case where
	we can deal with asymmetric kernels, since the regions $\{x>0\}$ and $\{x<0\}$ are clearly separated.
\end{enumerate}

%
%
%
%
%


\section{Critical kernels}
\label{sect:critical}
\setcounter{equation}{0}

We assume now that $J$ is symmetric and that
$$J(y)=\e^{-|y|\omega(y)}\,, \text{ where }\lim_{|y|\to\infty} \omega(y)=\ell\,.$$
Hence $J$ satisfies \eqref{eq:hyp3:J} with $\beta_0=\ell$.
We want to show that the estimate remains valid even if $H$ is not finite everywhere.
To this aim, we have to study the Hamilton-Jacobi equation more carefully.
In this case, the domain of definition of the Hamiltonian $H$ is exactly:
$\mathrm{dom}(H)=B_\ell$, that is, $H=+\infty$ outside $B_\ell$.
For simplicity, we will first assume that $\ell=\beta_0=1$,
the adaptations for other values being straightforward. Then we shall give
the general result in Theorem \ref{theorem:critical:kernel}.

\subsection{Hamilton-Jacobi equation with nonfinite hamiltonian}

We study  the equation $u_t+H(Du)=0$, posed in the cylinder $Q:=B_1\times(0,\infty)$
(although most of the results of this section would hold also for more general cylinders). Here, we assume that the hamiltonian $H$ is
infinite in the complement of $B_1$, which is the main difficulty. Following \cite{Barles90}, we begin by constructing a
new equation which is equivalent in the viscosity sense to $u_t+H(Du)=0$, the main interest being that it
allows us to prove comparison and analyze the initial trace of solutions.

On the parabolic boundary
$$\partial_P Q=(\partial B_1\times[0,\infty))\cup (B_1\times\{t=0\}\big)\,,$$
we impose a continuous boundary condition $f$ in the viscous sense.
More precisely, we consider the following problem:
\begin{equation}\label{eq:HJ0}
       \begin{cases}
               u_t+H(Du)=0 & \text{in }B_1\times(0,\infty)\,,\\
               u(x,t)=f(x,t) & \text{on }\partial B_1\times[0,\infty)\,,\\
               u(x,0)=f(x,0) & \text{in }B_1\,.
       \end{cases}
\end{equation}

\begin{definition}
       Given $f\in\mathrm{C}(\partial_P Q)$, we say that an upper semi-continuous function $u$
       is a viscosity subsolution of \eqref{eq:HJ0} if for any smooth function $\varphi$
       such that $u-\varphi$ reaches a maximum at $(x_0,t_0)$ we have:
       \begin{eqnarray*}
               (x_0,t_0)\in Q &\Rightarrow& \varphi_t + H(D\varphi)\leq0\quad \text{at }(x_0,t_0)\\
               (x_0,t_0)\in\partial_P Q &\Rightarrow & \min\Big(\varphi_t + H(D\varphi)\,;\,u-f\Big)\leq0\quad  \text{at }(x_0,t_0)\,.
       \end{eqnarray*}
\end{definition}

The same definition holds with reversed inequalities (and min instead of max) for an upper semi-continuous viscosity
subsolution. And finally:

\begin{definition}
       A locally bounded function $u:Q\to\R$ is a viscosity solution of \eqref{eq:HJ0} if its upper semi-continuous envelope
       is a supersolution and its lower semi-continuous envelope is a subsolution of~\eqref{eq:HJ0}.
\end{definition}

Let us mention that in the case when $H$ is finite everywhere, solutions take on the initial data in a
classical way. But since here some data may not be compatible with the fact that $\mathrm{dom}(H)=B_1$, this
implies that a boundary/initial layer appears, and this is precisely the phenomenon we are facing. In order to
understand this layer, we need first to reinterpret the equation with a new Hamiltonian:

\begin{proposition}
       Subsolutions and supersolutions of \eqref{eq:HJ0} are also subsolutions and supersolutions
       (in the viscous sense) of the equation:
       \begin{equation}\label{eq:modifiedHJ}
       \max\Big(u_t+H(Du)\,;\,|Du|-1\Big)=0\,.
       \end{equation}
       with the same data on the parabolic boundary.
\end{proposition}

\begin{proof}
       Let $u$ be a viscosity subsolution and consider a smooth test function $\varphi$ such that $u-\varphi$
       has a maximum at $(x_0,t_0)$. We assume for simplicity that $(x_0,t_0)\in Q$ (the argument being
       similar if it is a boundary point). Then since by definition
       $$\partial_t \varphi(x_0,t_0)+H(D\varphi)(x_0,t_0)\leq0\,,$$
       we have necessarily $H(D\varphi)<+\infty$, so that $|D\varphi|\leq1$ and thus $u$ satisfies (in the viscous sense)
       also the inequation
       $$\max\Big(\partial_t u+H(Du)\,;\,|Du|-1\Big)\leq0\,.$$
       Now if $v$ is a supersolution and $\varphi$ is such that $v-\varphi$ has a minimum at $(x_0,t_0)$,
       then $$\partial_t \varphi(x_0,t_0)+H(D\varphi)(x_0,t_0)\geq0$$
       implies that $$\max\Big(\partial_t \varphi(x_0,t_0)+H(D\varphi)(x_0,t_0)\,;\,|D\varphi(x_0,t_0)|-1\Big)\geq0\,,$$
       which implies that $v$ is a supersolution of \eqref{eq:modifiedHJ}.
\end{proof}

\begin{proposition}\label{prop:comp-HJ}
       Let $u$ be a viscosity subsolution and $v$ be a viscosity supersolution of
       $$\max\Big(u_t+H(Du)\,;\,|Du|-1\Big)=0\,.$$
       with $u\leq v$ on the parabolic boundary $\partial_P Q$. Then $u\leq v$.
\end{proposition}

\begin{proof}
       Formally speaking, if one fixes $\mu\in(0,1)$ and considers a maximum point $(x_0,t_0)$
       of $\mu u-v$, then if $|Dv|(x_0,t_0)\geq1$, since by definition $\mu Du=Dv$ at $(x_0,t_0)$,
       it comes that $|Du|(x_0,t_0)\geq1/\mu>1$ which is impossible since $u$ is a viscosity subsolution.
       So we have both $|Dv|,|Du|<1$ at $(x_0,t_0)$ and we do the comparison as always, using standard viscosity techniques.
Now let us be more precise.

       We fix $\mu\in(0,1)$, and $T>0$ and consider a point
       $(x_0,y_0,t_0,s_0)\in \overline{B_1}^2\times[0,T]^2$ where
       $$\Phi:(x,y,t,s)\mapsto\mu u(x,t)-v(y,s)-\frac{|x-y|^2}{\eps^2}-\frac{|t-s|^2}{\beta^2}-\frac{C}{T-t}$$
       reaches its maximum. We assume that it is an interior point otherwise,
       using the boundary values one obtains immediately $\Phi\leq0$ in $B_1^2\times(0,T)^2$ which is
       what we want.

       Fixing one variable, since $(x,t)\mapsto \Phi(x,y_0,t,s_0)$ reaches its maximum at $(x_0,t_0)$ one may
       consider the following test function for $u$ at $(x_0,t_0)$:
       $$\phi_1(x,t):=\frac{1}{\mu}\Big(v(y_0,s_0)+\frac{|x-y_0|^2}{\eps^2}+\frac{|t-s_0|^2}{\beta^2}+
       \frac{C}{T-t}\Big)\,.$$
       Indeed, if we denote by $p:=2(x_0-y_0)/\eps^2$, $q:=2(t_0-s_0)/\beta^2$, it comes
       \begin{equation}\label{eq:comp-HJ-1}
       \max\Big(\frac{\mu^{-1}C}{(T-t)^2}+\mu^{-1}q+H(\mu^{-1}p)\,;\,|\mu^{-1}p|-1\Big)\leq0\,.
       \end{equation}
       On the other hand, for $v$ we use at $(y_0,s_0)$ the test function
       \begin{equation*}
       \phi_2(s,t):=\mu u(x_0,t_0)-\frac{|x_0-y|^2}{\eps^2}-\frac{|t_0-s|^2}{\beta^2}\,,
       \end{equation*}
       which leads to
       \begin{equation}
         \label{eq:comp-HJ-2}
        \max\Big(q+H(p)\,;\,|p|-1\Big)\geq0\,.
       \end{equation}

       If we assume that $|p|\geq1$ then $\mu^{-1}|p|\geq\mu^{-1}>1$ which is impossible from \eqref{eq:comp-HJ-1}.
       So, both $|p|$ and $\mu^{-1}|p|$ are less than 1 and then the proof follows standard arguments of
       viscosity solutions: we can combine \eqref{eq:comp-HJ-1} and \eqref{eq:comp-HJ-2}, getting rid of the max
       which gives (after multiplying the first inequality by $\mu$):
       \begin{equation}\label{eq:comp-HJ-3}
       \frac{C}{(T-t)^2}+\mu H(\mu^{-1}p)-H(p)\leq0\,.
       \end{equation}
       We claim that $h(\mu):=\mu H(\mu^{-1}p)-H(p)\geq0$ for any $p\in\R^N$ and $\mu\in(0,1)$, which leads to a contradiction with~\eqref{eq:comp-HJ-3}, so that an interior maximum of $\Phi$
       is impossible.  Hence, $\Phi\leq0$ in $(B_1)^2\times(0,T)^2$ and since $\beta,\eps,C,T,\mu$ are arbitrary,
       we finally conclude that $u\leq v$ in $B_1\times(0,\infty)$.

       To end the proof, let us check the claim:
       using the convexity of $H$, one gets
       $$h'(\mu)=-\frac{1}{\mu}\Big((\mu^{-1}p)\cdot D H(\mu^{-1}p)-H(\mu^{-1}p)\Big)\leq0\,,$$
       and since $h(1)=0$, we see that $h\geq0$ for $\mu\in(0,1)$.
\end{proof}

\begin{proposition}
       \label{prop:HJ.initial.trace}
       Let $u$ be a viscosity solution of $\max\Big(u_t+H(Du)\,;\,|Du|-1\Big)=0\,,$
       and $u=f$ on $\partial_P Q$. Then the initial trace of $u$, $u(0)$ verifies:
       \begin{equation}
         \label{eq:obstacle}
         \max\Big(u(0)-f;\,|Du(0)|-1\Big)=0\,.
       \end{equation}
\end{proposition}

\begin{proof}
       Let $u$ be a subsolution of the equation with boundary data $f$. Then at $t=0$,
       $$\min\Big(\max(u_t+H(Du),|Du|-1)\,;\,u-f\Big)\leq0\,.$$
       We take a test function $\varphi(x,t)=C_\eps t + |x-x_0|^2/\eps^2$ such that
       $$\max(\varphi_t+H(D\varphi),|D\varphi|-1)\geq0\,.$$
       This is always possible if $\eps$ is small enough, so that this implies
       $u(0)\leq f$. Thus:
       $$\max\Big(u(0)-f;\,|Du(0)|-1\Big)\leq0\,.$$

       Now, we consider a supersolution $u$ and take $\varphi(x)$ such that $f-\varphi$ has a minimum
       at $x_0$. Then for any $C>0$,
       $$\psi:(x,t)\mapsto u(x,0)-\varphi(x)+Ct$$
       has a maximum at $t=0$, $x=x_0$. Using $\psi$ as test-function, we obtain
       $$\max\Big(\max(-C+H(D\varphi(0)),|D\varphi(0)|-1)\,;\,u-f\Big)\geq0\,.$$
       For $C$ big enough, we have $-C+H(D\varphi(0))<0$ so that there remain two possibilities:

       {\sc Case 1: } $|D\varphi(0)-1|<0$, which implies $u-f\geq0$ so that: $$\max\Big(\varphi(0)-f;\,|D\varphi(0)|-1\Big)\geq0\,.$$
       {\sc Case 2: } $|D\varphi(0)-1|\geq0$, which implies $$\max\Big(\varphi(0)-f;\,|D\varphi(0)|-1\Big)\geq0\,.$$

       So if $u$ is a solution, both inequalities give the equality.
\end{proof}

\begin{remark}
  {\rm  Equation~\eqref{eq:obstacle} can be understood as an obstacle problem: both $|Du(0)|\leq 1$, $u(0)\leq f$ and
  $(u(0)-f)(|Du(0)|-1)=0$.}
\end{remark}
\subsection{Back to the estimate of $I$}

As in Section \ref{sect:theoretical}, we define for any $A>0$:
\begin{equation}
       I_R^A:=-\frac{1}{R}\ln \big(\e^{-RA}+I\big)\, ,\qquad\ \overline{I}^A:= \limsup_{R\to\infty}I_R^A
       \, ,\ \qquad \underline{I}^A:= \liminf_{R\to\infty}I_R^A\,.
\end{equation}

Taking liminf and limsup and a test function, we see that for any $A>0$, $\overline{I}^A$ and
$\underline{I}^A$ are respectively sub and supersolutions of
$$\max\Big(u_t+H(Du)\,;\,|Du|-1\Big)=0\,,\quad I^A(0)=A\,.$$

Consequently, using \eqref{prop:comp-HJ}, we obtain that $\overline{I}^A=\underline{I}^A$, so that
as $R\to+\infty$, all the sequence $I^A_R$ converges to the unique solution $I^A$ of the problem.
Then as $A\to+\infty$, $I^A\to I$ which satisfies the equation with $I(0,x)=\dist(x;\partial B_1)$.

Now we have to identify the limit $I$, and to do so we have to study some properties of this specific Lagrangian.

\begin{lem}
	The Lagrangian $L$ satisfies the following properties:
	$$|DL(q)|<1\quad\text{and}\quad L(q)\sim |q|\text{ as }|q|\to\infty\,.$$
\end{lem}

\begin{proof}
	Since by definition,
	$$L(q)=\sup_{p\in\R^N}\{p\cdot q-H(p)\}=\sup_{|p|<1}\{p\cdot q-H(p)\}\,,$$
	and $H(p)\to+\infty$ as $|p|\to1$, this implies that the sup is attained at some $|p_0(q)|<1$.
	On the other hand, a simple calculus shows that $DL(q)=p_0(q)$, so that indeed, for any $q$,
	$|DL(q)|<1$. This also implies a first basic estimate: $L(q)\leq |q|$.
	To get the equivalent, we first bound $H(p)$. Let us first notice that
	$$H(p)=\Hess(p)+f(p)=\Hess(p)+O(1)\,,$$
	indeed the singular and differential parts of the hamiltonian remain bounded in the set $\{|p|<1\}$,
	as well as $Df(p)$.	Now, for any $|p|<1$,
	$$H(p)\leq f(p)+\int_{\R^N}\e^{(|p|-1)|y|}\dy\leq f(p)+c(N)\int_0^\infty\e^{(|p|-1)r}r^{N-1}\d r=
	f(p)+\frac{c(N)(N-1)!}{(1-|p|)^N}\,.$$
	Hence,
	$$L(q)\geq\sup_{|p|<1}\big\{p\cdot q - \frac{c(N)(N-1)!}{(1-|p|)^N}-f(p)\big\}\,,$$
	and this sup is attained for $p_0=p_0(q)$ satisfying the equation:
	$$q=|p|c(N)(N-1)!\big(1-|p|\big)^{-N-1}\frac{p}{|p|}+Df(p)\,.$$
	Thus, as $|q|\to\infty$, necessarily $|p_0|\to1$ and since $p_0$ and $q$ point in the same direction,
	$$L(q)\geq\sup_{|p|<1}\big\{p\cdot q - \frac{c(N)(N-1)!}{(1-|p|)^N}-f(p)\big\}\sim p_0(q)\cdot q \sim |q|\,.$$
	Since we have seen that $L(q)\leq|q|$, we conclude that
	$$L(q)\sim|q|\quad\text{as}\quad |q|\to \infty\,.$$
\end{proof}

Since $J$ is assumed to be symmetric, so are $H$ and $L$ and we write $\underline{L}(|x|)=L(x)$.

\begin{proposition}
	The solution of \eqref{eq:HJ0} with initial data $I(0,x)=\dist(x;\partial B_1)$ and $I=0$ on the boundary
	is given by the Lax-Oleinik formula:
       \begin{equation*}
               I(x,t)=t\underline{L}\Big(\frac{\dist(x;\partial B_1)}{t}\Big)\,.
       \end{equation*}
\end{proposition}
\begin{proof}
	Since $|DL|<1$, then $|DI|<1$ so that the compatibility condition is always fulfilled and
	the equation holds everywhere. Now we take a look at the initial data. Since $L(q)\sim|q|$,
	this implies:
	$$t\underline{L}\Big(\frac{\dist(x;\partial B_1)}{t}\Big)\to \dist(x;\partial B_1)\text{  as  }t\to0\,,$$
	so that indeed the Lax-Oleinik formula gives a solution. Since the viscosity solution is unique,
	this ends the proof.
\end{proof}

The reader will easily check that if $\ell=\beta_0\neq1$, then all the results of this section remain valid
and then $L(q)\sim\beta_0|q|$ as $|q|\to\infty$. Moreover, Lemma~\ref{lemma:H:singular} and
Lemma~\ref{lemma:reduce:symmetric} are also valid in the present case since $|p|$ remains bounded: as we have seen,
$H(p)=\Hess(p)+O(1)$ as $|p|=\beta_0$.

Hence we may write down a more general result for possibly non-symmetric and singular kernels:

\begin{theorem}\label{theorem:critical:kernel}
	Let $J$ be a kernel satisfying \eqref{eq:hyp:J} and \eqref{eq:hyp3:J}. In particular,
	$J$ can be asymmetric and have a singularity	at the origin.
	Then for any $\theta\in(0,1)$, we have the following estimate as $R\to\infty$: for any $\theta\in(0,1)$
	$T>0$,
	\begin{equation}
        \sup_{|x|\leq\theta R\atop 0\leq t\leq TR}|u-u_R|(x,t)\leq \e^{-{(1-\theta)}{\beta_0} R+o(R)}
    \end{equation}
\end{theorem}

\begin{proof} We skip the details since this is the same as for Proposition
\ref{theo:compact:support}: we first reduce the estimate
to symmetric kernel by comparison, putting a symmetric kernel above $J$ with the same $\beta_0$
in \eqref{eq:hyp3:J} and then we wipe out the possible
singularities and differential terms. Actually, the proof is even simpler since
since those terms remain bounded in the set $\{|p|<\beta_0\}$, hence only the exponential part of the Hamiltonian
plays a role in the estimate.
\end{proof}

\begin{remark}{\rm As $\beta_0\to 0$, the estimate gets worse. Indeed, this means that the kernel tends to
	behave slower than an exponential and we are facing a problem of fat tails (like a power decay),
	that this method cannot handle.}
\end{remark}

\section{Further results and comments}
\setcounter{equation}{0}
\label{sect:comments}

\subsection{Explicit bounds}\label{subsect:explicit:estimates}

In some particular cases, we already gave concrete estimates in~\cite{BrandleChasseigne08-1}, by studying directly the Hamiltonian $H$.
Computing the function $\mathcal{K}^{-1}$,  we recover here these estimates under the following form:
\begin{equation*}
     \sup_{|x|<M, 0<t<1}|u-u_R|(x,t)\leq \e^{-R\,\mathcal{K}^{-1}(\ln R)(1+o(1))}
\end{equation*}

\begin{table}[ht]\label{table:behaviour}
  \begin{center}
  \begin{tabular}
  {c@{\hskip2cm}c@{\hskip2cm}c}
  \hline\\
  {\bf Kernel}& $\mathcal{K}^{-1}(z) $& $\sup|u-u_R|$\\[12pt]
  \hline
  \\
  $J(y)=\ind{{|y|\leq \rho}}(y)$ & $\displaystyle z\rho^{-1}$& $\e^{-R\ln(R)/\rho}$ \\[6pt]
  $J(y)=\e^{-\e^{|y|}}$& $ z(\ln z)^{-1}$&$\e^{-R(R\ln R)}$\\[6pt]
  $J(y)=\e^{-|y|^\alpha}$,\  $\alpha> 1$& $ z^{(\alpha-1)/\alpha}$&$\e^{-R(\ln R)^{(\alpha-1)/\alpha}}$\\[6pt]
  $J(y)=\e^{-|y|\ln |y|}$&$\ln z$&$\e^{-R\ln\ln R}$\\[6pt]
  $J(y)=\e^{-\alpha|y|},\  \alpha>0$& $\alpha$& $\e^{-\alpha R}$\\[6pt]
  \end{tabular}
\end{center}
  \caption{Behaviour of $\mathcal{K}^{-1}$}
\end{table}

Table \ref{table:behaviour} collects some known asymptotic behaviors of $\mathcal{K}^{-1}$ and
$\sup(u-u_R)$. Notice that since $\mathcal{K}$ is superlinear (except in the critical case), then
$\mathcal{K}^{-1}(z)$ is defined for $z>0$ big enough.
Most of the calculations are straightforward, we only sketch the case $J(y)=\e^{\e^{-y}}$: in this case,
$\mathcal{K}(y)\sim |y|\ln|y|$ and if we define $z=|y|\ln|y|$ we get $\ln z=\ln|y|+\ln \ln |y|\sim\ln|y|$. Hence
$z\sim |y|\ln z$, which implies $|y|\sim z/\ln z$. Let us also mention that in the case $J(y)=\e^{-\alpha|y|}$,
$\mathcal{K}^{-1}(z)=\alpha$ in the sense of graphs.

\begin{remark}
  {\rm As we have seen, the presence of a singularity at the origin does not modify the behaviour of $\mathcal{K}^{-1}$, so by instance, if we multiply by
  $  |y|^{-2}  $
  any of the previous kernels we obtain the same estimate.}
\end{remark}

\subsection{Non-symmetric kernels in one space dimension}\label{subsect:asymetric:1D}

As we have seen, it is in general difficult to give an explicit behaviour of
$\sup|u-u_R|$ for non-symetric kernels, unless we compare with a symetric one.
But in the case $N=1$, the regions where $\{x>0\}$ and $\{x<0\}$ are clearly seperated
so that more precise estimates can be given according to the tails
 of $J$ at $-\infty$ and $+\infty$ which can be different. We shall just illustrate this
in an explicit case which concerns the most extreme cases we can cover: on the one
side we have a compactly supported kernel, while on the other side, we have an exponential decay.

\begin{proposition}
    \label{prop:J:nonsymmetric}
    Let us assume that $N=1$ and $J$ is given by:
	$$J(y):= \frac{1}{2}\e^{-|y|}\ind{y<0}+\frac{1}{2}\ind{y\geq0}\,.$$
	Then the associated Lagrangian satisfies:	
	$$L(q)\mathop{\sim}_{q\to+\infty} q \ln q\,,\quad\text{and}\quad
	L(q)\mathop{\sim}_{q\to-\infty} q\,.$$
\end{proposition}
\begin{proof}
	A straightforward calculus shows that the Hamiltonian is defined in
	the region $\{p>-1\}$ and that:
	$$H(p)=\frac{\e^p}{2p}-\frac{1}{2p}+\frac{1}{2(p+1)}-1\,.$$
	For $q>0$, we calculate the Lagragian as follows:
	$$L(q)=\sup_{p>-1}\{pq-H(p)\}=\sup_{p>0}\{pq-H(p)\}\,.$$
	Indeed, if $q>0$, then $pq>0$ only for $p>0$ and we know the sup is nonnegative so that
	it has to be attained for $p>0$. With this remark, the estimate is just the same as for
	the case when $J(y)=\frac{1}{2}\ind{y\geq0}$. The same remark holds when $q<0$: the sup
	is attained in the region $-1<p<0$ and the behaviour is given by then exponential decay of
	$J$.
\end{proof}

Then we are able to use Theorem \ref{thm:est-IR} in a more precise way (we consider $x$ fixed in a bounded
domain for simplicity):
\begin{proposition}
	Let $J$ be defined as above. Then
	\begin{eqnarray*}
  	\sup_{0<x<M,0<t<1} |u-u_R|(x,t)&\leq& \e^{-R\ln R (1+o(1))}\\
  	\sup_{-M<x<0,0<t<1} |u-u_R|(x,t)&\leq& \e^{-R(1+o(1))}\\
  	\end{eqnarray*}
\end{proposition}
\begin{proof}
	We just come back to the expression of $I_\infty$:
	$$
		I_\infty(x/R,t/R)=\min_{y\in\partial B_R}\frac{t}{R}\,L\Big(\frac{Rx-y}{t}\Big)\,.
	$$
	So, if $1<x<M$, whether $y=-R$ or $y=+R$, we always have $Rx-y\geq (M-1)R\to+\infty$.
	Hence the min of $L$ is attained for $y=-R$ and we recover the behaviour of $L(q)$ for
	$q=R(M+1)\to\infty$, that is, a $R\ln R$ behaviour.
	On the other hand, if $-M<x<-1$, the min is attained for $y=-R$ and
	we get the linear behaviour of $L(q)$ for $q\to-\infty$.
\end{proof}

\begin{figure}[ht]
  \begin{center}
  \includegraphics[width=8cm]{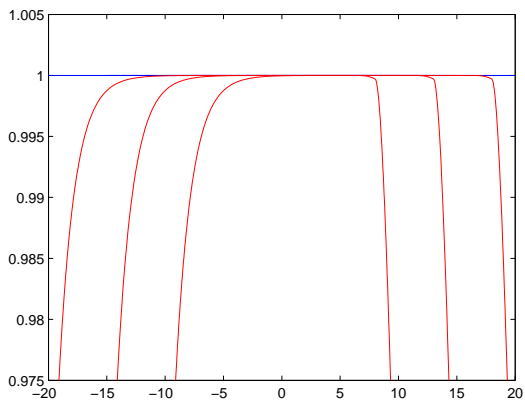}
  \caption{Convergence for non-symmetric kernel.}
  \label{fig:nosimetrico}
  \end{center}
\end{figure}

In Figure~\ref{fig:nosimetrico} we plot for $u_0=1$ the approximations $u_R$ for $R=10$, $15$, $20$ and non symmetric $J$ as in Proposition~\ref{prop:J:nonsymmetric}. This illustrates the different rate of convergence whether $x<0$ or $x>0$.
\subsection{KPP-type results}

In this section, we briefly explain how our results allow to treat a non-local version of
the KPP-problem (Kolmogorov-Petrovskii-Piskounov) associated to equation~\eqref{eq:0} with the classical monostable $u(1-u)$-term. For simplicity we shall just explain this on the following equation:
\begin{equation*}\label{eq:KPP-0}
	\partial_t u - (J*u-u) = u(1-u)\quad\text{in}\quad\R^N\times[0,T]\,,
\end{equation*}
with a continuous initial data $u_0(x)=g(x)$, $0\leq g\leq 1$.  Existence of solutions with
initial data $0\leq g\leq 1$ may be obtained for instance by Perron's method.

The interested reader will find further references about this equation and traveling waves
in the works of J.~Coville and L.~Dupaigne \cite{CovilleDupaigne}.

Now, in order to study convergence of
$u$ to the equilibrium states $0$ and $1$ for large $x$ and $t$, the following scaling
is widely used:
\begin{equation*}
	u^\eps(x,t)=u\Big(\frac{x}{\eps},\frac{t}{\eps}\Big)\,.
\end{equation*}
It turns out here that formally, $u^\eps$ satisfies the equation
\begin{equation*}
	\partial_t u^\eps - \frac{1}{\eps}(J_\eps*u^\eps-u^\eps) = \frac{u^\eps(1-u^\eps)}{\eps}
	\quad\text{in}\quad\R^N\times[0,T]\,,
\end{equation*}
with $J_\eps(x)=\eps^NJ(x/\eps)$, so that one may use exactly the same method as was used
in Section~\ref{sect:theoretical} with $\eps$ playing the role of $1/R\to0$.

We may thus combine the techniques of \cite{BarlesSouganidis} with the ones we used in Section \ref{sect:theoretical} to handle the convergence of the non-local term, which give us some estimates at which convergence to the states $u=0$ and $u=1$ occur. In fact, once we know how to deal with the non-local terms of the equation, the rest of the proof only follows \cite{BarlesSouganidis}, this is why we only sketch a proof below. Here we denote by $H$ the following Hamiltonian:
$$H(p):=\int_{\R^N}\Big(\e^{p\cdot y}-1\Big)J(y)\dy\,,$$
where it is assumed that $J\in\mathrm{C}(\R^N)$ and satisfies \eqref{eq:hyp3:J} for some $\beta_0>0$.

\begin{theorem} Let $G_0=\{g(x)=0\}$ and $G_1=\{g(x)=1\}$. Then the following results hold:
	\noindent\emph{(i)} Let $I^\eps_0=-\eps\ln (u^\eps)$. Then $I^\eps_0\to I_0$, the solution
	of the variational inequality:
	\begin{equation*}
		\min\Big(\frac{\partial I_0}{\partial t}+H(D I_0) + 1\,,\,I_0\Big)=0\quad\text{in}
		\quad \R^N\times(0,\infty)\,,
	\end{equation*}
	and
	\begin{equation*}
		I_0(x,0)=\begin{cases}
		       	+\infty & \text{if }x\in G_0\,,\\
		       	0 & \text{if }x\in\R^N\setminus G_0\,.
		       \end{cases}
	\end{equation*}

	\noindent\emph{(ii)} Let $I^\eps_1:=-\eps\ln(1-u^\eps)$. Then $I^\eps_1\to I_1$
	locally uniformly in the set $\mathcal{C}=\big\{I_0(x,t)=0\big\}$, where
	$$	
    \begin{cases}
		\dfrac{\partial I_1}{\partial t}+H(D I_1) + 1=0&\quad\text{in}\quad \mathop{\rm int}\mathcal{C}\\[6pt]
		I_1(x,t)=0 &\quad{on}\quad \partial\mathcal{C}\cap\{t>0\}\\[6pt]
		 I_1(x,0)=+\infty\ind{G_1}(x).
		\end{cases}
$$
\end{theorem}

\begin{proof}[Sketch of proof]
	For (i), it turns out that $u^\eps$ satisfies the equation
\begin{equation*}\label{eq:KPP-1}
	\partial_t u^\eps - \frac{1}{\eps}(J_\eps*u^\eps-u^\eps) = \frac{u^\eps(1-u^\eps)}{\eps}
	\quad\text{in}\quad\R^N\times[0,T]\,,
\end{equation*}
with initial data $u^\eps(x,0)=g(x/\eps)$, where the rescaled kernel is $J_\eps(x)=\eps^NJ(x/\eps)$.

Then we make the log-transform
\begin{equation*}
	I^{\eps,A}_0(x,t)=-\eps\ln \Big(u^\eps(x,t)+\exp\big(-\frac{A}{\eps}\big)\Big)
\end{equation*}
which satisfies
\begin{equation*}
	\partial_t I^{\eps,A}_0 +
	\int\Big(\e^{-\frac{1}{\eps}\big\{I^{\eps,A}_0(x+y\eps)-I^{\eps,A}_0(x)\big\}}-1\Big)J(y)\dy
	= -\frac{u^\eps(1-u^\eps)}{u^\eps + \exp(-A/\eps)}
\end{equation*}
Passage to the limit in the left-hand side is done exactly as in Section~\ref{sect:theoretical} while
handling the right-hand side follows exactly from KPP classical techniques: first notice that
by construction, $I^{\eps,A}_0\geq0$. Then, if in the limit
$I_0^{A}(x,t)=\lim_{\eps\to0}I^{\eps,A}_0>0$, this means that $u^\eps\to0$ so that
clearly as $\eps\to0$, the right-hand side converges to $1$.

For (ii), we set similarly
\begin{equation*}
	I^{\eps,A}_1(x,t)=-\eps\ln \Big(1-u^\eps(x,t)+\exp\big(-\frac{A}{\eps}\big)\Big)
\end{equation*}
The first step consists in proving that $\overline{I^A_1}:=\limsup_{\eps\to0}I^{\eps,A}_1$
and $\underline{I^A_1}:=\liminf_{\eps\to0}I^{\eps,A}_1$ are respectively sub- and super-solutions
of the variational inequality:
\begin{eqnarray*}
	\max\Big(\partial_t I+H(D I)+1,I-\psi^A\Big)&=&0\quad\text{in}\quad\R^N\times(0,\infty)\,,\\
	I(x,0)&=&A\cdot\ind{G_1}(x),
\end{eqnarray*}
where $\psi^A(x,t)=1$ if $(x,t)\in\mathcal{C}$ and $0$ otherwise. Direct comparison
between $\overline{I^A_1}$ and $\underline{I^A_1}$ cannot be derived here since no
information on the regularity of $\mathcal{C}$ is available. The final result then
follows from a representation formula for $I^A_1$. We refer to \cite{BarlesSouganidis} for the details.
\end{proof}

Then, coming back to the original variables, one can obtain explicit exponential convergence rates,
which follows from our study of the asymptotic behaviour of the Lagrangian associated to $H$
(as was done in Sections \ref{sect:compact}, \ref{sect:intermediate}, \ref{sect:critical}).

\subsection{Relation with optimal existence results}

We would like to add another final comment on a related subject.
As was said, by estimating $\sup_{B_R}|u-u_R|(x,T)$, we are measuring the total amount of
processes that can escape the box $B_R$ between $t=0$ and $t=T$. Another way of understanding
this is that we are somehow estimating the Green kernel associated to the equation.

Thus in \cite{BrandleChasseigneFerreira09}, the authors together with R. Ferreira are deriving similar
estimates but in the context of optimal initial data, which is also a way of measuring
the behaviour of the kernel at infinity. Hence it is not so surprising that similar estimates
appear, even if they are obtained through a totally different method.

For instance, it turns out that if $J$ is compactly supported,
the optimal class of existence for $u_t=J\ast u -u$ in $\R^N$ consists of initial data satisfying:
$$|u_0(x)|\leq C\e^{|x|\ln |x|}\,,$$
hence we recover a $R\ln R$ estimate for the Green function, typical of compactly supported kernels.
We refer to \cite{BrandleChasseigneFerreira09} for more results in this direction.

{\small \begin{center}
\textsc{Acknowledgments}

Both authors partially supported by project MTM2008-06326-C02-02 (Spain).
\end{center}}


\end{document}